\documentclass[11pt]{article}

\usepackage{amssymb, amsthm, amsmath}
\usepackage{mathtools}
\usepackage{derivative}
\usepackage{cleveref}
\renewcommand{\phi}{\varphi}

\newcommand*{\real}{\mathbb{R}}
\newcommand*{\complex}{\mathbb{C}}
\newcommand*{\quaternion}{\mathbb{H}}

\newcommand*{\Ogroup}{\mathrm{O}}
\newcommand*{\SOgroup}{\mathrm{SO}}
\newcommand*{\GLgroup}{\mathrm{GL}}
\newcommand*{\Pin}{\mathrm{Pin}}
\newcommand*{\Spin}{\mathrm{Spin}}
\newcommand*{\Mat}{\mathrm{Mat}}

\newcommand*{\B}{\mathrm{B}}

\newcommand*{\Id}{\mathrm{Id}}

\newcommand*{\clifford}{\mathrm{Cl}}
\newcommand*{\lipschitz}{\Gamma}

\newcommand*{\gradeinv}[1]{\hat{#1}}
\newcommand*{\Gradeinv}[1]{\widehat{#1}}
\newcommand*{\reversal}[1]{\tilde{#1}}
\newcommand*{\Reversal}[1]{\widetilde{#1}}
\newcommand*{\conj}[1]{\bar{#1}}
\newcommand*{\tran}{\mathsf T}

\let\widebar\relax
\newcommand{\widebar}[1]{\mkern 1.5mu\overline{\mkern-1.5mu#1\mkern-1.5mu}\mkern 1.5mu}
\newcommand*{\Conj}[1]{\widebar{#1}}

\DeclareMathOperator{\sgn}{sgn}

\begingroup
\catcode`|=\active
\gdef\@changevert{\def|{\:\delimsize\vert\:}}
\endgroup

    \DeclarePairedDelimiter\absval{\lvert}{\rvert}

    \DeclarePairedDelimiterX{\inner}[2]{\langle}{\rangle}{#1, #2}

    \newcommand\DeclarePairedDelimiterV[3]{\expandafter\DeclarePairedDelimiterX\csname#1\endcsname[1]#2#3{\mathcode`|="8000\@changevert ##1}}
    \DeclarePairedDelimiterV{set}{\{}{\}}
\AtBeginDocument{
    \newtheorem{theorem}{Theorem}
    \newtheorem{lemma}[theorem]{Lemma}
    \newtheorem{proposition}[theorem]{Proposition}
    \newtheorem{corollary}[theorem]{Corollary}
    \theoremstyle{definition}
    \newtheorem{definition}[theorem]{Definition}
    
}

\RequirePackage[backend=biber,
style=alphabetic
]{biblatex}

\bibliography{References}

\title{Conformal Invariance of Clifford Monogenic Functions in the Indefinite Signature Case}
\author{Chen Liang\footnote{Undergraduate student at the University of California, 1 Shields Ave, Davis, CA 95616}
and Matvei Libine\footnote{Department of Mathematics, Indiana University, Rawles Hall, 831 East 3rd St, Bloomington, IN 47405}}

\begin{document}

\maketitle

\begin{abstract}
We extend constructions of classical Clifford analysis to the case of
indefinite non-degenerate quadratic forms.
Clifford analogues of complex holomorphic functions
-- called monogenic functions -- are defined by means of
the Dirac operators that factor a certain wave operator.
%The subject of this paper is the property of invariance of monogenic functions
%under conformal or M\"obius transformations.
One of the fundamental features of quaternionic analysis is the
invariance of quaternionic analogues of holomorphic function
%-- called regular functions --
under conformal (or M\"obius) transformations.
A similar invariance property is known to hold in the context of
Clifford algebras associated to positive definite quadratic forms.
We generalize these results to the case of Clifford algebras
associated to all non-degenerate quadratic forms.
This result puts the indefinite signature case on the same footing as the
classical positive definite case.
%Clifford algebras generalize many features of complex numbers and quaternions.
%Quaternionic analogues of holomorphic functions -- called regular functions --
%are known to be invariant under conformal or M\"obius transformations.
%It is also known that a similar invariance property holds in the context of
%Clifford algebras associated to positive definite quadratic forms.
%In this project, we generalize these results to the case of Clifford algebras
%associated to all non-degenerate quadratic forms.
%This approach puts the indefinite signature case on the same footing as the
%classical positive definite case.
\end{abstract}


\section{Introduction}

Many results of complex analysis have analogues in quaternionic analysis.
In particular, there are analogues of complex holomorphic functions called
(left and right) regular functions.
Some of the most fundamental features of quaternionic analysis are
the quaternionic analogue of Cauchy's integral formula
(usually referred to as Cauchy-Fueter formula) and
the invariance of quaternionic regular functions under
conformal (or M\"obius) transformations.
For modern introductions to quaternionic analysis see, for example,
\cite{Sudbery.QuaternionicAnalysis, colombo2004analysis}.

Complex numbers $\complex$ and quaternions $\quaternion$ are special cases of
Clifford algebras.
(For elementary introductions to Clifford algebras see, for example,
\cite{Chevalley.AlgebraicTheorySpinor, Garling.CliffordIntroduction}.)
There is a further extension of complex and quaternionic analysis called
Clifford analysis.
For Clifford algebras associated to positive definite quadratic forms on real
vector spaces, Clifford analysis is very similar to complex and quaternionic
analysis (see, for example,
\cite{brackx1982clifford, gilbert1991clifford, delanghe1992clifford}
and references therein).
Furthermore, Ryan has initiated the study of Clifford analysis in the
setting of complex Clifford algebras \cite{Ryan.ComplexifiedClifford}.

Let \(\clifford(V)\) be the universal Clifford algebra associated to a
real vector space \(V\) with non-degenerate quadratic form \(Q\).
The Dirac operator \(D\) on \(V\) is introduced by means of factoring
the wave operator associated to \(Q\).
Then the Clifford monogenic functions \(f: V \to \clifford(V)\) are defined
as those satisfying \(Df=0\).
An analogue of Cauchy's integral formula for such functions was established in
\cite{Libine_Sandine_2021}.
In this paper we focus on another important feature of Clifford analysis
-- conformal invariance of monogenic functions.
If \(A = \begin{psmallmatrix} a & b \\ c & d \end{psmallmatrix}\)
is a Vahlen matrix producing a conformal (or M\"obius) transformation on \(V\)
\begin{equation}\label{ctransf-intro}
    x \mapsto (ax+b)(cx+d)^{-1},
\end{equation}
and \(f: V \to \clifford(V)\) is a monogenic function, then it is known
in the case of positive definite quadratic forms \(Q\) that
\begin{equation}\label{eq:result_preview}
  J_A(x) f(A x) = \frac{(cx+d)^{-1}}
  {\absval{(x\reversal{c} + \reversal{d})(cx+d)}^{n/2-1}} f(Ax)
\end{equation}
is also monogenic \cite{Ryan.SingularitiesClifford, Bojarski1989}.
We extend this result to all non-degenerate quadratic forms \(Q\)
having arbitrary signatures.
While the formula in the mixed signature case visually appears the same,
its meaning is slightly different. First of all, the transformation
\eqref{ctransf-intro} takes place on the conformal closure of the vector
space \(V\), and it is different from the one-point compactification of \(V\).
Secondly, one needs to revisit the definition of the monogenic functions in
this context and choose the "right" Dirac operator \(D\).
The choice of signs in our definition of the operator \(D\) is compatible with
\cite{Bures-Soucek, Budinich1988, Libine_Sandine_2021}.
Furthermore, we argue that this is the most natural choice of the Dirac
operator that makes the results valid, while other choices would not work.
Interestingly, Bure\v s-Sou\v cek \cite{Bures-Soucek} arrive at the same
Dirac operator from a different perspective -- invariance under the action
of the group $\Spin(p,q)$.

These two features of Clifford analysis with indefinite signature
-- the analogue of Cauchy's integral formula and the conformal invariance
of monogenic functions put the indefinite signature case on the same
footing as the classical positive definite case.

The paper is organized as follows.
Section \ref{review} is a review of relevant topics.
We start with the conformal transformations and
the conformal compactification of \(V\).
Then we discuss Clifford algebras and associated groups --
such as Lipschitz, pin and spin groups -- and their actions on \(V\)
by orthogonal linear transformations.
After this we introduce the basics of Clifford analysis:
the Dirac operator \(D\) and the left/right monogenic functions.
We show that our definition of \(D\) is independent of the choice of
basis of \(V\) (Lemma \ref{D-basis-independence}) and give an alternative
basis-free way of defining \(D\) (Proposition \ref{D-basis-independence}).
In Section \ref{Vahlen} we introduce Vahlen matrices and explain
the relation between Vahlen matrices and conformal transformations on \(V\).
The results of this section are well known, especially in the positive
definite case (see, for example,
\cite{Ahlfors.MobiusClifford, Vahlen1902, Maks.CliffordMobius,
  Cnops.VahlenIndefinite} and references therein), and we mostly follow
\cite{Maks.CliffordMobius}.
In Section \ref{invariance} we prove our main result (Theorem \ref{main})
that the function \eqref{eq:result_preview} is left monogenic.
We also state a similar formula for right monogenic functions.
We emphasize that our proof is independent of the signature of
the quadratic form \(Q\).
Finally, in Appendix \ref{appendix} we discuss what would happen if
the Dirac operator were defined differently and show that a different
choice of signs in the definition of the operator \(D\) would result
in the loss of the conformal invariance.
In \cite{Libine_Sandine_2021} the same choice of signs
was motivated by completely different reasons.

\thanks{
This research was made possible by the Indiana University, Bloomington,
Math REU (research experiences for undergraduates) program,
funded by NSF Award \#2051032.
We would like to thank Professor Dylan Thurston for organizing and
facilitating this program.
We would also like to thank Ms. Mandie McCarthy and Ms. Amy Bland for
their administrative work, the various professors for their talks,
and the other REU students for their company.
Finally, we would like to thank the reviewers for their feedback
and suggesting additional references.
}

\section{Conformal Transformations, Clifford Algebras and Monogenic Functions}  \label{review}

In this section we briefly review the fundamentals of
conformal transformations, Clifford algebras and
Clifford analysis and establish our notations.

\subsection{The Conformal Compactification of $\real^{p,q}$}

We start with a review of the conformal compactification \(N(V)\)
of a vector space \(V \simeq \mathbb{R}^{p,q}\)
and the action of the indefinite orthogonal group
\(\Ogroup(V\oplus\real^{1,1}) \simeq \Ogroup(p+1,q+1)\) on \(N(V)\)
by conformal transformations.
The reader may wish to refer to, for example,
\cite{Schottenloher.CFT, kobayashi1995transformation}
for more detailed expositions of the results.

Let $V$ be an $n$-dimensional real vector space,
and $Q$ a quadratic form on $V$.
Corresponding to $Q$, there is a symmetric bilinear form $B$ on $V$.
The form $Q$ can be diagonalized: there exist an orthogonal basis
$\{e_1,\dots, e_n\}$ of $(V,Q)$ and integers $p$, $q$ such that
\begin{equation}  \label{V-basis}
Q(e_j)=
\begin{cases} 1 & 1 \leq j \leq p;\\
-1 & p+1 \leq j \leq p+q; \\
0 & p+q < j \leq n. \end{cases}
\end{equation}
By the Sylvester's Law of Inertia, the numbers $p$ and $q$ are independent of
the basis chosen. The ordered pair $(p,q)$ is called the signature of $(V,Q)$.
From this point on we restrict our attention to non-degenerate quadratic forms
$Q$, in which case $p+q=n$.
If $Q$ is such a form with $q=0$ or $p=0$, then it is positive or negative
definite respectively.

Associated with the quadratic form is the indefinite orthogonal group
consisting of all invertible linear transformations of $V$ that preserve $Q$:
$$
\Ogroup(V)= \bigl\{ T \in \GLgroup(V) ;\:
Q(Tv)=Q(v)\text{ for all } v \in V \bigr\}.
$$

We frequently identify $(V,Q)$ with the generalized Minkowski space
$\real^{p,q}$, which is the real vector space $\real^{p+q}$ equipped with
indefinite quadratic form
$$
Q(u) = (x_1)^2+\dots+(x_p)^2-(x_{p+1})^2-\dots-(x_{p+q})^2.
$$
The symmetric bilinear form associated to $Q$ is
\begin{align*}
B(u,v)&=
\tfrac12 \bigl[ Q(u+v) - Q(u) - Q(v) \bigr] \\
&=x_1y_1+\dots+x_py_p-x_{p+1}y_{p+1}-\dots -x_{p+q}y_{p+q}.
\end{align*} 
In this context, it is common to write $\Ogroup(p,q)$ for the indefinite
orthogonal group $\Ogroup(\real^{p,q})$.

By a {\em conformal transformation} of $(V,Q)$, we mean a smooth mapping
$\phi: U \to V$, where $U \subset V$ is a non-empty connected open subset,
such that the pull-back
\begin{equation}  \label{conformal-transf-condition}
\phi^*B(u,v) = \Omega^2 \cdot B(u,v)
\end{equation}
for some smooth function $\Omega: U \to (0,\infty)$.
(Note that we do not require conformal transformations to be orientation
preserving.)

The {\em conformal compactification} $N(V)$ of $V$ is constructed using
an ambient vector space $V \oplus \real^{1,1}$.
The $2$-dimensional space $\real^{1,1}$ has a standard basis
\(\set{e_-, e_+}\), and the quadratic form $Q$ on $V$ naturally extends to
a form $\hat Q$ on $V \oplus \real^{1,1}$ so that
$$
\hat Q(v + x_-e_- + x_+e_+) = Q(v) - (x_-)^2 + (x_+)^2.
$$
By definition, $N(V)$ is a quadric in the real projective space:
$$
N(V) = \bigl \{ [\xi] \in P(V \oplus \real^{1,1}) ;\: \hat{Q}(\xi)=0 \bigr\}.
$$
Then $N(V)$ has a natural conformal structure inherited from
$V \oplus \real^{1,1}$.
More concretely, $V \oplus \real^{1,1}$ contains the product of unit
spheres $S^p \times S^q$, and the quotient map
$\varpi: V \oplus \real^{1,1} \setminus \{0\} \twoheadrightarrow
P(V \oplus \real^{1,1})$
restricted to $S^p \times S^q$ induces a smooth $2$-to-$1$ covering
$$
\pi : S^p\times S^q \to N(V),
$$
which is a local diffeomorphism.
The product $S^p \times S^q$ has a semi-Riemannian metric coming from the
product of metrics --  the first factor having the standard metric from
the embedding  $S^p \subset \real^{p+1}$ as the unit sphere,
and the second factor $S^q$ having the negative of the standard metric.
Then the semi-Riemannian metric on $N(V)$ is defined by declaring
the covering $\pi : S^p\times S^q \to N(V)$ to be a local isometry.
We also see that $N(V)$ is compact, since it is the image of a compact set
$S^p \times S^q$ under a continuous map $\pi$.

The embedding map $\iota: V \hookrightarrow N(V)$ can be described as follows.
We embed $V$ into the null cone (without the origin)
\begin{equation}  \label{null-cone}
{\mathcal N} = \bigl \{ \xi \in V \oplus \real^{1,1} ;\: \xi \ne 0,\:
\hat{Q}(\xi)=0 \bigr\}
\end{equation}
via
\begin{equation}  \label{conf-embedding}
\hat\iota(v) = v + \frac{1 + Q(v)}2 e_- + \frac{1 - Q(v)}2 e_+,
\end{equation}
and compose $\hat\iota$ with the quotient map
$V \oplus \real^{1,1} \setminus \{0\} \twoheadrightarrow P(V \oplus \real^{1,1})$.

\begin{proposition}
The quadric $N(V)$ is indeed a conformal compactification of $V$.
That is, the map $\iota : V \to N(V)$ is a conformal embedding,
and the closure $\overline{\iota(V)}$ is all of $N(V)$.
\end{proposition}

The indefinite orthogonal group $\Ogroup(V \oplus \real^{1,1})$ acts on
$V \oplus \real^{1,1}$ by linear transformations.
This action descends to the projective space $P(V \oplus \real^{1,1})$
and preserves $N(V)$.
Thus, for every $T \in \Ogroup(V \oplus \real^{1,1})$, we obtain a
transformation on $N(V)$, which will be denoted by $\psi_T$.

\begin{theorem}
For each $T \in \Ogroup(V \oplus \real^{1,1})$, the map $\psi_T: N(V) \to N(V)$
is a conformal transformation and a diffeomorphism.
If $T, T' \in \Ogroup(V \oplus \real^{1,1})$, then $\psi_T = \psi_{T'}$
if and only if $T' = \pm T$.
\end{theorem}

The following proposition essentially asserts that the restrictions of $\psi_T$
to $V$ are compositions of parallel translations, rotations,
dilations and inversions.

\begin{proposition}\label{types}
For each $T \in \Ogroup(V \oplus \real^{1,1})$, the conformal transformation
$\psi_T$ can be
written as a composition of $\psi_{T_j}$, for some
$T_j \in \Ogroup(V \oplus \real^{1,1})$,
$j=1,\dots,k$, such that each $\iota^{-1} \circ \psi_{T_j} \circ \iota$
is one of the following types of transformations:
\begin{enumerate}
\item
  parallel translations $v \mapsto v+b$, where $b \in V$;
\item
  orthogonal linear transformations $v \mapsto Tv$,
  where $T \in \Ogroup(p,q)$;
\item
  dilations $v \mapsto \lambda v$, where $\lambda>0$;
\item
  the inversion $v \mapsto v/Q(v)$.
\end{enumerate}
\end{proposition}

When $p+q>2$, these are {\em all} possible conformal transformations of
$N(V)$. In fact, an even stronger result is true.

\begin{theorem}  \label{conformal-continuation-thm}
Let $p+q>2$. Every conformal transformation $\phi :U \to V$,
where $U \subset V$ is any non-empty connected open subset,
can be uniquely extended to $N(V)$, i.e. there exists a unique conformal
diffeomorphism $\hat{\phi} : N(V) \to N(V)$ such that
$\hat{\phi}(\iota(v))=\iota(\phi(v))$
for all $v \in U$.

Moreover, every conformal diffeomorphism $N(V) \to N(V)$
must be of the form $\psi_T$, for some $T \in \Ogroup(V \oplus \real^{1,1})$.
Thus, the group of all conformal transformations $N(V) \to N(V)$
is isomorphic to $\Ogroup(V \oplus \real^{1,1})/\{\pm \Id\}$.
\end{theorem}

We emphasize that by the {\em group of conformal transformations of $V$}
we mean all conformal transformations $N(V) \to N(V)$, and these are not
required to be orientation preserving.

We conclude this subsection with a classification of points in
the conformal compactification \(N(V)\).
The points in \(N(V)\) are represented by the equivalence classes
\([v + x_- e_- + x_+ e_+]\) in the real projective space
\(P(V \oplus \real^{1,1})\) such that \(Q(v) - (x_-)^2 + (x_+)^2 = 0\),
where \(v \in V\), \(x_-, x_+ \in \real\).
We can distinguish between the following three classes:
\begin{itemize}
\item
\(x_- = \frac12 (1 + Q(v))\) and \(x_+ = \frac12 (1 - Q(v))\).
These points can be identified with vectors in \(V\) through
\eqref{conf-embedding}.
%This class is closed under parallel translations, orthogonal transformations,
%and dilations.
%The subset of points that represents vectors in \(V\) that are not null is
%closed under inversion.

\item
\(Q(v) = 0\), \(x_- = -\frac12\) and \(x_+ = \frac12\).
Such points represent the image of the null vectors \(v \in V\)
(i.e., vectors with \(Q(v)=0\)) under the unique extension to \(N(V)\)
of the inversion map \(v \mapsto v/Q(v)\).

\item
\(Q(v) = x_- = x_+ = 0\) and \(v\neq 0\).
We can think of these points as the limiting points of the lines
generated by the non-zero null vectors \(v \in V\).
\end{itemize}

\subsection{Clifford Algebras and Associated Groups}

For elementary introductions to Clifford algebras see, for example,
\cite{Chevalley.AlgebraicTheorySpinor, Garling.CliffordIntroduction}.

Let $V$ be a real vector space together with a non-degenerate
quadratic form $Q$.
Recall the standard construction of the universal Clifford algebra associated
to $(V,Q)$ as a quotient of the tensor algebra.
We start with the tensor algebra over $V$,
\begin{equation*}
\bigotimes V=\real \oplus V\oplus (V\otimes V)\oplus (V \otimes V \otimes V)
\oplus \dots,
\end{equation*}
consider a set of elements of the tensor algebra
\begin{equation}  \label{ideal}
S=\{v\otimes v+Q(v): v\in V\} \quad \subset \bigotimes V,
\end{equation}
and let $(S)$ denote the ideal of $\bigotimes V$ generated by the elements
of $S$. Then the universal Clifford algebra $\clifford(V)$ associated to
$(V,Q)$ can be defined as a quotient
\begin{equation*}  %\label{quotientdef}
\clifford(V) =\bigotimes V/(S).
\end{equation*}
We note the sign convention in \eqref{ideal} chosen for the generators
of the ideal $(S)$ is not standard in the literature, but it seems to be
more prevalent in the context of Clifford analysis.
We fix an orthogonal basis $\{e_1,e_2,\dots,e_n\}$ of $V$ satisfying
\eqref{V-basis}, and let $e_0 \in \clifford(V)$ denote the multiplicative
identity of the algebra.
Thus $\clifford(V)$ is a finite-dimensional algebra over $\real$ generated by
$e_0,e_1,\dots,e_n$.
We consider subsets $A \subseteq \{1,\dots, n\}$.
If $A$ has $k>0$ elements,
\begin{equation*}
A=\{i_1,i_2,\dots,i_k\}\subseteq \{1,\dots, n\}
\end{equation*}
with $i_1<i_1< \dots <i_k$, define
\begin{equation*}
e_A=e_{i_1i_2\dots i_k}=e_{i_1}\otimes e_{i_2} \otimes \dots \otimes e_{i_k}
\end{equation*}
or, more precisely, $e_A$ is the image of this tensor product in
$\clifford(V)$.
If $A$ is empty, we set $e_{\emptyset}$ be the identity element $e_0$.
These elements
\begin{equation}  \label{Clifford-basis}
\bigl\{e_A: A \subseteq \{1,2,\dots, n\} \bigr\}
\end{equation}
form a vector space basis of $\clifford(V)$ over $\real$.
In particular, $V$ can be identified with the vector subspace of
$\clifford(V)$ spanned by $e_1,e_2,\dots,e_n$, and we frequently do so.
%treat $V$ as a subspace of $\clifford(V)$.

If $u$ and $v$ are elements of $V \subset \clifford(V)$, then
\begin{equation}  \label{xy+yx}
uv+vu = (u+v)^2 - u^2 - v^2 = -Q(u+v) + Q(u) + Q(v) = - 2 B(u,v).
\end{equation}
In particular, if $u$ and $v$ are orthogonal in $V$, $uv+vu=0$.

If \(\mathcal A\) is a unital algebra, by a {\em Clifford map} we mean
an injective linear mapping \(j: V \to {\mathcal A}\) such that
\(1 \notin j(V)\) and \(j(v)^2 = -Q(v) 1\) for all \(v \in V\).
The algebra \(\clifford(V)\) is {\em universal} in the sense that each
such Clifford map \(j: V \to {\mathcal A}\) has a unique extension to
a homomorphism of algebras \(\clifford(V) \to {\mathcal A}\).

There are three important involutions defined on the Clifford algebra
\(\clifford(V)\).
\begin{itemize}
\item
The {\em grade involution} \(\hat{\:}: \clifford(V) \to \clifford(V)\)
is induced by the Clifford map \(V \to \clifford(V)\) defined by
\(x \mapsto -x\). Then, for vectors \(v_1, v_2, \dots, v_k \in V\), we have
        \begin{equation*}
            \Gradeinv{v_1v_2\cdots v_k} = (-1)^k v_1v_2\cdots v_k.
        \end{equation*}
        Thus, in the standard basis \eqref{Clifford-basis},
        the grade involution is given by
        \begin{equation*}
            \Gradeinv{e_A} \mapsto (-1)^{\absval{A}} e_A
        \end{equation*}
        where \(\absval{A}\) is the cardinality of \(A\).

\item
The {\em reversal} \(\reversal{\:} : \clifford(V) \to \clifford(V)\)
is the transpose operation on the level of tensor algebra which descends to
the Clifford algebra because the two sided ideal \((S)\) generated by
\eqref{ideal} is preserved.
The transpose operation is defined on each summand \(V^{\otimes k}\) of the
tensor algebra \(\bigotimes V\) as
\begin{equation*}
v_1 \otimes \cdots \otimes v_k \mapsto v_k \otimes \cdots \otimes v_1.
\end{equation*}
Hence, for vectors \(v_1, \dots, v_k \in V\), we have
\begin{equation*}
\Reversal{v_1 \cdots v_k} = v_k \cdots v_1.
\end{equation*}
In the standard basis \eqref{Clifford-basis}, the reversal is given by
\begin{equation*}
\Reversal{e_A} = (-1)^{\frac{\absval{A}(\absval{A}-1)}{2}} e_A.
\end{equation*}
The reversal is an algebra anti-homomorphism:
\(\Reversal{ab} = \reversal{b} \reversal{a}\).

\item
The {\em Clifford conjugation} is defined as the composition of
the grade involution and the reversal (note that these two involutions commute).
We denote the Clifford conjugation by \(\conj{a} = \gradeinv{\reversal{a}}\).
In the standard basis \eqref{Clifford-basis}, the Clifford conjugation is
given by
\begin{equation*}
\Conj{e_A} = (-1)^{\frac{\absval{A}(\absval{A}+1)}{2}} e_A.
\end{equation*}
Like the reversal, the Clifford conjugation is an algebra anti-homomorphism:
\(\Conj{ab} = \conj{b} \conj{a}\).
\end{itemize}

We frequently use that the grade involution, the reversal,
and the Clifford conjugation
all commute with taking the multiplicative inverse in \(\clifford(V)\):
\begin{equation*}
\Gradeinv{(a^{-1})} = \gradeinv{a}^{-1}, \quad
\Reversal{(a^{-1})} = \reversal{a}^{-1}, \quad
\Conj{(a^{-1})} = \conj{a}^{-1},
\qquad a \in \clifford(V).
\end{equation*}
Also note that a vector \(v\) is invertible as an element of \(\clifford(V)\)
if and only if it is not null, i.e., \(Q(v) \ne 0\), and
\(v^{-1} = -v/Q(v)\).

An essential feature of Clifford algebras is that the orthogonal
transformations on \(V\) can be realized via the so-called twisted
adjoint action by elements of \(\clifford(V)\).
Let \(v \in V\) be a vector that is not null, then every \(x\in V\) splits
into \(x = x^\perp + \lambda v\) with \(x^\perp\in V\) orthogonal to \(v\).
%Explicitly,
%\begin{equation*}
%x^{\perp} = x - \frac{B(x,v)}{Q(v)} v \quad\text{and}\quad \lambda = \frac{B(x,v)}{Q(v)}.
%\end{equation*}
The map \(\sigma_v: \clifford(V) \to \clifford(V)\) defined as
\begin{equation*}
\sigma_v(x) = -v x v^{-1}
\end{equation*}
preserves the space \(V\) and produces a reflection in the direction of \(v\).
Indeed,
\begin{equation*}
\sigma_v(x) = \sigma_v(x^\perp + \lambda v)
= -v x^{\perp}v^{-1} - v (\lambda v) v^{-1}
= x^\perp - \lambda v.
\end{equation*}
Combining such reflections produces various orthogonal transformations on \(V\).

Recall the Cartan-Dieudonn\'e Theorem (for more details see, for example,
\cite{Garling.CliffordIntroduction}).

\begin{theorem}[Cartan-Dieudonn\'e]
Let \(V\) be a vector space with non-degenerate quadratic form,
then every orthogonal linear transformation on \(V\) can be expressed
as the composition of at most \(\dim V\) reflections in the direction of
vectors that are not null.
\end{theorem}

By the Cartan-Dieudonn\'e theorem, all orthogonal linear transformations on
\(V\) can be expressed using a twisted adjoint action
\begin{equation*}
    \sigma_a : x \mapsto a x \gradeinv{a}^{-1}
\end{equation*}
where \(a \in \clifford(V)\) is a product of vectors in \(V\) that are not null.

\begin{definition}
Let \(\clifford^\times(V)\) denote the set of invertible elements in
\(\clifford(V)\).
The {\em Lipschitz group}\footnote{Sometimes this group is called Clifford-Lipschitz group or Clifford group.}
\(\lipschitz(V)\) consists of elements in
\(\clifford^{\times}(V)\) that preserve vector space \(V\) under the
twisted adjoint action:
\begin{equation*}
\lipschitz(V) = \set{a \in \clifford^\times(V) ;\:
\sigma_a(x) = a x \gradeinv{a}^{-1} \in V \text{ for all } x \in V}.
\end{equation*}
\end{definition}

The following facts about the Lipschitz group \(\lipschitz(V)\) are well known.
For each \(a\in \lipschitz(V)\), \(\sigma_a\) restricts to a linear
transformation of \(V\) that preserves the quadratic form \(Q(x)\).
Moreover, we have an exact sequence
\begin{equation}  \label{Lipschitz-exact-seq}
  1 \longrightarrow \real^{\times} \longrightarrow \lipschitz(V)
  \xlongrightarrow{\sigma} \Ogroup(V) \longrightarrow 1.
\end{equation}

\begin{definition}
The {\em pin group} \(\Pin(V)\) and the {\em spin group} \(\Spin(V)\)
are defined as
\begin{equation}  \label{pin-def}
\begin{split}
\Pin(V) &= \set{a \in \lipschitz(V) ;\: \conj{a} a = \pm 1}, \\
\Spin(V) &= \set{a \in \lipschitz(V) ;\: \conj{a} a = \pm 1,\ \gradeinv{a} = a}.
\end{split}
\end{equation}
\end{definition}

The pin group and the spin group are double covers of the orthogonal group
and special orthogonal group respectively.
This can be summarized by the exact sequences
\begin{equation*}
  1 \longrightarrow \set{\pm 1} \longrightarrow \Pin(V)
  \xlongrightarrow{\sigma} \Ogroup(V) \longrightarrow 1,
\end{equation*}
\begin{equation*}
  1 \longrightarrow \set{\pm 1} \longrightarrow \Spin(V)
  \xlongrightarrow{\sigma} \SOgroup(V) \longrightarrow 1.
\end{equation*}

\begin{proposition}
The Lipschitz group \(\lipschitz(V)\) is the subgroup of \(\clifford^\times(V)\)
generated by \(\real^\times\) and the invertible vectors in \(V\).
Moreover, every element in \(\lipschitz(V)\) can be expressed as a product
of \(\real^\times\) and at most \(n = \dim V\) invertible vectors.
\end{proposition}

\begin{corollary}  \label{cor:lipschitz_conjugate}
If \(a \in \clifford(V)\), then \(a\in \lipschitz(V)\) if and only if
\(a\conj{a} \in \real^\times\) and \(a x \reversal{a} \in V\) for all
\(x \in V\).
\end{corollary}

\begin{proof}
Suppose first \(a\in \lipschitz(V)\), then we can write
\(a = \alpha v_1 v_2 \cdots v_k\) as a product of vectors
\(v_1, \dots, v_k \in V\) that are not null and \(\alpha \in \real^{\times}\).
It is clear that \(a \conj{a} = \alpha^2 Q(v_1)\cdots Q(v_k) \in \real^\times\).
For each \(x\in V\), since \(\conj{a} = a^{-1} (a \conj{a})\), we have
\begin{equation}  \label{axa-reversal}
  a x \reversal{a} = a x \gradeinv{\conj{a}}
  = a x \gradeinv{a}^{-1} (a \conj{a}) \in V.
\end{equation}
Conversely, if \(a\conj{a} \in \real^\times\), then \(a \in \clifford^\times(V)\).
And \(a x \reversal{a} \in V\) for all \(x \in V\) implies that
\begin{equation*}
% a x \gradeinv{a}^{-1} = - \Gradeinv{a x \gradeinv{a}^{-1}} =
\gradeinv{a} x a^{-1}
= - \Conj{\frac{a x \reversal{a}}{a \conj{a}}} \in V
\qquad \text{for all } x \in V.
\end{equation*}
Applying the grade involution, we see that
\(a x \gradeinv{a}^{-1} = \gradeinv{a} x a^{-1} \in V\) for all \(x\in V\).
Thus, \(a \in \lipschitz(V)\).
\end{proof}

\begin{corollary}\label{cor:lipschitz_conjugate_2}
 If \(a \in \lipschitz(V)\), then \(\conj{a} x a \in V\) for all \(x\in V\).
\end{corollary}

\begin{proof}
Write \(a = \alpha v_1 v_2 \cdots v_k\) as a product of vectors
\(v_1, \dots, v_k \in V\) that are not null and \(\alpha \in \real^{\times}\),
then we have
\begin{equation*}
\conj{a} x a = (-1)^{k} \alpha^2 v_k\cdots v_1 x v_1 \cdots v_k
= (-1)^{k} \alpha^2 v_k\cdots v_1 x \Reversal{v_k\cdots v_1} \in V
\end{equation*}
for all \(x\in V\) because \(v_k\cdots v_1 \in \lipschitz(V)\).
\end{proof}

%Returning to the conformal transformations, observe that
%\(\omega = e_1 e_2 \cdots e_n \)
%(the product of the orthonormal basis vectors) acts on \(V\) via
%\(\sigma_{\omega}\) as the total reflection \(v \mapsto -v\).
%Thus, Proposition \ref{prop:conformal_continuation} can be rephrased as
%follows.
%\begin{proposition}
%When \(\dim V >2\), the group of all conformal diffeomorphisms
%\(N(V) \to N(V)\) is isomorphic to
%\(\Pin(V\oplus \real^{1,1})/\set{\pm 1, \pm \omega}\).
%\end{proposition}

\subsection{Monogenic Functions}
\label{ss:monogenic_function}

In this subsection we review the definitions of the Dirac operator and
monogenic functions in the context of Clifford analysis.
Some comprehensive works on Clifford analysis include
\cite{brackx1982clifford, gilbert1991clifford, delanghe1992clifford}
as well as references therein.

We identify the vector space \(V\) with \(\real^{p,q}\) and introduce
the Dirac operator \(D\) on \(\real^{p,q}\)
\begin{equation}  \label{Dirac-operator}
D = e_1 \pdv{}{x_1} + \cdots + e_p \pdv{}{x_p} - e_{p+1} \pdv{}{x_{p+1}}
- \cdots - e_{p+q} \pdv{}{x_{p+q}}.
\end{equation}
This operator can be applied to functions on the left and on the right.
The choice of signs in \eqref{Dirac-operator} is compatible with
\cite{Bures-Soucek, Budinich1988, Libine_Sandine_2021},
and it is discussed further in Appendix  \ref{appendix}.

If $f$ is a twice-differentiable function on \(\real^{p,q}\) with values in
$\clifford(\real^{p,q})$ or a left $\clifford(\real^{p,q})$-module and
$g$ is a twice-differentiable function on \(\real^{p,q}\) with values in
$\clifford(\real^{p,q})$ or a right $\clifford(\real^{p,q})$-module,
\begin{equation*}
  DDf = -\Box_{p,q}f, \quad
  gDD = -\Box_{p,q}g, \quad \text{where} \quad
  \Box_{p,q} = \sum_{j=1}^p \pdv[order=2]{}{x_j}
  - \sum_{j=p+1}^{p+q} \pdv[order=2]{}{x_j}.
\end{equation*}
%\begin{equation*}
%\begin{split}
%  DDf &= -\pdv[{order=2}]{f}{x_1} - \cdots - \pdv[order=2]{f}{x_p}
%  + \pdv[order=2]{f}{x_{p+1}} + \cdots + \pdv[order=2]{f}{x_{p+q}},  \\
%  gDD &= -\pdv[{order=2}]{g}{x_1} - \cdots - \pdv[order=2]{g}{x_p}
%  + \pdv[order=2]{g}{x_{p+1}} + \cdots + \pdv[order=2]{g}{x_{p+q}}.
%\end{split}
%\end{equation*}
Thus, we can think of the Dirac operator \(D\) as a factor of the wave
operator \(\Box_{p,q}\) on \(\real^{p,q}\).

\begin{definition}
Let $U \subseteq \real^{p,q}$ be an open set, and $\cal M$ a left
$\clifford(\real^{p,q})$-module. A differentiable function $f: U \to {\cal M}$
is called {\em left monogenic} if
\begin{equation*}
D f = e_1 \pdv{f}{x_1} + \cdots + e_p \pdv{f}{x_p} - e_{p+1} \pdv{f}{x_{p+1}}
- \cdots - e_{p+q} \pdv{f}{x_{p+q}} =0
\end{equation*}
at all points in $U$.

Similarly, let ${\cal M}'$ be a right $\clifford(\real^{p,q})$-module, then
a differentiable function $g: U \to {\cal M}'$
is called {\em right monogenic} if
\begin{equation*}
  gD = \pdv{g}{x_1} e_1 + \cdots + \pdv{g}{x_p} e_p - \pdv{g}{x_{p+1}} e_{p+1}
- \cdots - \pdv{g}{x_{p+q}} e_{p+q} =0
\end{equation*}
at all points in $U$.
\end{definition}

We often regard \(\clifford(\real^{p,q})\) itself as a
\(\clifford(\real^{p,q})\)-module and speak of left or right monogenic
functions with values in \(\clifford(\real^{p,q})\).
Modules over \(\clifford(\real^{p,q})\) are well understood since
the seminal paper by Atiyah, Bott and Shapiro \cite{Atiyah-Bott-Shapiro};
they treated modules over complex and real Clifford algebras.
The result is saying that such modules are completely reducible and that
there is a classification of the irreducible ones.
More detailed exposition of this topic can be found in, for example,
\cite{Bures-Soucek, Budinich1988}.
Hence, without loss of generality one can consider functions with values
in one of these irreducible \(\clifford(\real^{p,q})\)-modules.

For the remainder of this subsection we discuss some properties of the Dirac
operator \(D\).

\begin{lemma}  \label{D-basis-independence}
  The Dirac operator \(D\) is independent of the choice of orthogonal
  basis of \(V \simeq \real^{p,q}\) satisfying \eqref{V-basis}.
\end{lemma}

\begin{proof}
  Let \(\set{e'_1,\dots,e'_n}\) be another basis of \(V\) satisfying
  \eqref{V-basis}, and let \(D'\) be the Dirac operator associated to
  this basis.
%We can rewrite \eqref{Dirac-operator} as
%\begin{equation*}
%D' = \begin{pmatrix} e'_1 & \cdots & e'_n \end{pmatrix}
%        \begin{pmatrix} 1_{p\times p} & 0_{p\times q} \\
%            0_{q\times p} & -1_{q\times q} \end{pmatrix}
%        \begin{pmatrix}\pdv{}{x'_1} \\ \vdots \\ \pdv{}{x'_n} \end{pmatrix}
%\end{equation*}
%and similarly for \(D\).
%\begin{equation*}
%\begin{split}
%    D &= \begin{pmatrix} e_1 & \cdots & e_n \end{pmatrix}
%        \begin{pmatrix} 1_{p\times p} & 0_{p\times q} \\
%            0_{q\times p} & -1_{q\times q}
%        \end{pmatrix}
%        \begin{pmatrix}\pdv{}{x_1} \\ \vdots \\ \pdv{}{x_n} \end{pmatrix},  \\
%D' &= \begin{pmatrix} e'_1 & \cdots & e'_n \end{pmatrix}
%        \begin{pmatrix} 1_{p\times p} & 0_{p\times q} \\
%            0_{q\times p} & -1_{q\times q}
%        \end{pmatrix}
%        \begin{pmatrix}\pdv{}{x'_1} \\ \vdots \\ \pdv{}{x'_n} \end{pmatrix}.
%\end{split}
%\end{equation*}  
There exists a matrix \(T \in \Ogroup(p,q)\) such that
\begin{equation*}
    \begin{pmatrix} e'_1 \\ \vdots \\ e'_n \end{pmatrix}
    = T \begin{pmatrix} e_1 \\ \vdots \\ e_n \end{pmatrix}
\qquad\text{and}\qquad
\begin{pmatrix} \pdv{}{x'_1} \\ \vdots \\ \pdv{}{x'_n} \end{pmatrix}
= T \begin{pmatrix} \pdv{}{x_1} \\ \vdots \\ \pdv{}{x_n} \end{pmatrix}.
\end{equation*}
Then
\begin{equation*}
\begin{split}
D' &= \begin{pmatrix} e'_1 & \cdots & e'_n \end{pmatrix}
        \begin{pmatrix} 1_{p\times p} & 0_{p\times q} \\
            0_{q\times p} & -1_{q\times q}
        \end{pmatrix}
        \begin{pmatrix}\pdv{}{x'_1} \\ \vdots \\ \pdv{}{x'_n} \end{pmatrix}  \\
&= \begin{pmatrix} e_1 & \cdots & e_n \end{pmatrix} T^{\tran}
        \begin{pmatrix} 1_{p\times p} & 0_{p\times q} \\
            0_{q\times p} & -1_{q\times q} \end{pmatrix}
   T \begin{pmatrix}\pdv{}{x_1} \\ \vdots \\ \pdv{}{x_n} \end{pmatrix} \\
&= \begin{pmatrix} e_1 & \cdots & e_n \end{pmatrix}
        \begin{pmatrix} 1_{p\times p} & 0_{p\times q} \\
            0_{q\times p} & -1_{q\times q} \end{pmatrix}
   \begin{pmatrix}\pdv{}{x_1} \\ \vdots \\ \pdv{}{x_n} \end{pmatrix}  \\
&= D.
\end{split}
\end{equation*}
%\begin{equation*}
%    \begin{pmatrix} e'_1 & \cdots & e'_n \end{pmatrix}
%    = \begin{pmatrix} e_1 & \cdots & e_n \end{pmatrix} T.
%\end{equation*}
%For the same matrix \(T\) we have
%\begin{equation*}
%\begin{pmatrix} \pdv{}{x'_1} & \cdots & \pdv{}{x'_n} \end{pmatrix}
%= \begin{pmatrix} \pdv{}{x_1} & \cdots & \pdv{}{x_n} \end{pmatrix} T.
%\end{equation*}
%Then
%\begin{equation*}
%\begin{split}
%D' &= \begin{pmatrix} e'_1 & \cdots & e'_n \end{pmatrix}
%        \begin{pmatrix} 1_{p\times p} & 0_{p\times q} \\
%            0_{q\times p} & -1_{q\times q}
%        \end{pmatrix}
%        \begin{pmatrix}\pdv{}{x'_1} \\ \vdots \\ \pdv{}{x'_n} \end{pmatrix}  \\
%&= \begin{pmatrix} e_1 & \cdots & e_n \end{pmatrix} T
%        \begin{pmatrix} 1_{p\times p} & 0_{p\times q} \\
%            0_{q\times p} & -1_{q\times q} \end{pmatrix}
%  T^{\tran} \begin{pmatrix}\pdv{}{x_1} \\ \vdots \\ \pdv{}{x_n} \end{pmatrix}\\
%&= \begin{pmatrix} e_1 & \cdots & e_n \end{pmatrix}
%        \begin{pmatrix} 1_{p\times p} & 0_{p\times q} \\
%            0_{q\times p} & -1_{q\times q} \end{pmatrix}
%   \begin{pmatrix}\pdv{}{x_1} \\ \vdots \\ \pdv{}{x_n} \end{pmatrix}  \\
%&= D.
%\end{split}
%\end{equation*}
\end{proof}

Furthermore, we can define the Dirac operator in a basis independent fashion.
Let \(V^*\) denote the dual space of \(V\), and treat the non-degenerate
symmetric bilinear form \(B\) on \(V\) as an element of \(V^* \otimes V^*\).
The form \(B\) induces isomorphisms \(V \simeq V^*\) and
\(V \otimes V \simeq V^* \otimes V^*\).
Denote by \(B^*\) the element in \(V \otimes V\) corresponding to \(B\).
Also, let \(V'\) denote the space of all translation-invariant first order
differential operators on \(V\).
Naturally, \(V\) and \( V'\) are isomorphic via
\begin{equation*}
V \ni \: v \longleftrightarrow \pdif{v} \: \in V',
\end{equation*}
where \(\pdif{v}\) denotes the directional derivative in the direction of
vector \(v\).
This isomorphism induces a map
\(d: V \otimes V \simeq V \otimes V' \to {\mathcal D}(V)\),
where \({\mathcal D}(V)\) denotes the space of all \(V\)-valued
translation-invariant first order differential operators on \(V\):
\begin{equation*}
d(u \otimes v) = u \pdif{v}.
\end{equation*}

\begin{proposition}  \label{D-basis-free}
  The image \(d(B^*) \in {\mathcal D}(V)\) is precisely the Dirac operator
  \(D\) defined by \eqref{Dirac-operator}.  
\end{proposition}

\begin{proof}
  Let \(\set{e_1,\dots,e_n}\) be any basis of \(V\), and let
  \(\set{e^*_1,\dots,e^*_n}\) be the dual basis of \(V^*\), then
  \begin{equation*}
%    B = \sum_{i,j} B(e_i,e_j) e^*_i \otimes e^*_j \in V^* \otimes V^*, \qquad
    B^* = \sum_{i,j} B^*(e^*_i,e^*_j) e_i \otimes e_j \:\in V \otimes V
    \quad \text{and} \quad
    d(B^*) =  \sum_{i,j} B^*(e^*_i,e^*_j) e_i \pdv{}{x_j} \:\in {\mathcal D}(V).
  \end{equation*}
  The matrix with entries \(B^*(e^*_i,e^*_j)\) is the inverse of the matrix
  with entries \(B(e_i,e_j)\).
  If \(\set{e_1,\dots,e_n}\) is an orthogonal basis of \(V\) satisfying
  \eqref{V-basis}, it is clear that \(d(B^*)\) equals
  the Dirac operator \(D\) defined by \eqref{Dirac-operator} relative to
  the basis \(\set{e_1,\dots,e_n}\).
\end{proof}

\begin{corollary}
  For any basis \(\set{e_1,\dots,e_n}\) of \(V\), let \(\B^{-1}\) be the
  inverse of the matrix with entries \(B(e_i,e_j)\), then
  \begin{equation}  \label{D-expression}
D = \begin{pmatrix} e_1 & \cdots & e_n \end{pmatrix} \B^{-1}
\begin{pmatrix}\pdv{}{x_1} \\ \vdots \\ \pdv{}{x_n} \end{pmatrix}
=  \sum_{\lambda,\mu=1}^n e_{\lambda} [\B^{-1}]_{\lambda\mu} \pdv{}{x_{\mu}}.
  \end{equation}
\end{corollary}

Expression \eqref{D-expression} essentially serves as the definition of
the Dirac operator in \cite{Budinich1988};
this definition is standard in physics literature.

\section{Vahlen Matrices}  \label{Vahlen}

We introduce Vahlen matrices, which are certain \(2\times 2\) matrices
with entries in \(\clifford(V)\), and explain the relation between
Vahlen matrices and conformal transformations on \(V\).
The results of this section are well known, especially in the positive
definite case
\cite{Ahlfors.MobiusClifford, Vahlen1902, Maks.CliffordMobius,
  Cnops.VahlenIndefinite} (and many other works).
The last two references are particularly relevant, since they deal with the
indefinite signature case.

\subsection{Algebra Isomorphism
  \(\clifford(V \oplus \real^{1,1}) \simeq \Mat(2, \clifford(V))\)}

We denote by \(\Mat(2, \clifford(V))\) the algebra of \(2\times 2\)
matrices with entries in \(\clifford(V)\).
The key ingredient is the \((1,1)\) periodicity of Clifford algebras
(see, for example,
\cite{delanghe1992clifford, Maks.CliffordMobius, Cnops.VahlenIndefinite}):

\begin{proposition}\label{prop:11_periodicity}
The Clifford algebra \(\clifford(V \oplus \real^{1,1})\) is isomorphic to
the matrix algebra \(\Mat(2, \clifford(V))\).
\end{proposition}

\begin{proof}
First, we construct a Clifford map
\(j: V \oplus \real^{1,1} \to \Mat(2, \clifford(V))\) by defining
\begin{equation*}
j(v) = \begin{pmatrix} v & 0  \\ 0 & -v \end{pmatrix}, \quad
j(e_-) = \begin{pmatrix} 0 & 1 \\ 1 & 0 \end{pmatrix}, \quad
j(e_+) = \begin{pmatrix} 0 & -1 \\ 1 & 0 \end{pmatrix}, \qquad
v \in V.
\end{equation*}
%for vectors \(v\in V\).
We need to check that \(j(x)j(x) = -\hat Q(x)\) for all
\(x \in V \oplus \real^{1,1}\).
Writing \(x = v + x_-e_- + x_+e_+\),
\begin{equation*}
j(x) = j(v + x_-e_- + x_+e_+)
= \begin{pmatrix} v & x_- - x_+  \\ x_- + x_+ & -v \end{pmatrix},
\end{equation*}
then
\begin{multline*}
j(x)j(x) = \begin{pmatrix} v & x_- - x_+  \\ x_- + x_+ & -v \end{pmatrix}^2  \\
= \begin{pmatrix} v^2 + (x_-)^2 - (x_+)^2 & 0 \\ 0 & v^2 + (x_-)^2 - (x_+)^2
\end{pmatrix}
= -\hat Q(x),
\end{multline*}
since \(v^2=-Q(v)\) by \eqref{xy+yx}.
This proves that \(j: V \oplus \real^{1,1} \to \Mat(2, \clifford(V))\)
is a Clifford map. By the universal property of Clifford algebras,
this map extends to an algebra homomorphism
\(\hat\jmath: \clifford(V \oplus \real^{1, 1}) \to \Mat(2, \clifford(V))\).

To show that $\hat\jmath$ is an isomorphism, note that for
\(a \in \clifford(V) \subset \clifford(V \oplus \real^{1, 1})\), we have
\begin{equation*}
\hat\jmath(a) = \begin{pmatrix} a & 0 \\ 0 & \gradeinv{a} \end{pmatrix},
\end{equation*}
and, for an arbitrary element
\(a + b e_- + c e_+ + d e_- e_+ \in \clifford(V \oplus \real^{1, 1})\),
where \(a,b,c,d \in \clifford(V)\), we have
\begin{equation}  \label{j-hat}
\begin{split}
\hat\jmath(a + b e_- &+ c e_+ + d e_- e_+)  \\
&=
\begin{pmatrix} a & 0 \\ 0 & \gradeinv{a} \end{pmatrix}
+ \begin{pmatrix} b & 0 \\ 0 & \gradeinv{b} \end{pmatrix}
\begin{pmatrix} 0 & 1 \\ 1 & 0 \end{pmatrix}
+ \begin{pmatrix} c & 0 \\ 0 & \gradeinv{c} \end{pmatrix}
\begin{pmatrix} 0 & -1 \\ 1 & 0 \end{pmatrix}
+ \begin{pmatrix} d & 0 \\ 0 & \gradeinv{d} \end{pmatrix}
\begin{pmatrix} 1 & 0 \\ 0 & -1 \end{pmatrix} \\
&= \begin{pmatrix} a + d & b - c \\
\gradeinv{b} + \gradeinv{c} & \gradeinv{a} - \gradeinv{d} \end{pmatrix}.
\end{split}
\end{equation}
It is clear that \(\hat\jmath\) is a bijection and hence an algebra isomorphism.
\end{proof}

From now on we use isomorphism \(\hat\jmath\) to identify
\(\clifford(V \oplus \real^{1, 1})\) with \(\Mat(2, \clifford(V))\).
The following result can be deduced from equation \eqref{j-hat} describing
this isomorphism.

\begin{lemma}\label{lmm:clifford_conjugate_matrix}
The three involutions of \(\clifford(V \oplus \real^{1, 1})\) in terms of
\(\Mat(2, \clifford(V))\) are given by
\begin{equation*}
\begin{split}
\text{the grade involution:} \qquad
\Gradeinv{\begin{pmatrix} a & b \\ c & d \end{pmatrix}}
&= \begin{pmatrix}\gradeinv{a} & -\gradeinv{b} \\
-\gradeinv{c} & \gradeinv{d} \end{pmatrix}, \\
\text{the reversal:} \qquad
\Reversal{\begin{pmatrix} a & b \\ c & d \end{pmatrix}}
&= \begin{pmatrix} \conj{d} & \conj{b} \\
  \conj{c} & \conj{a} \end{pmatrix}, \\
\text{the Clifford conjugation:} \qquad
\Conj{\begin{pmatrix} a & b \\ c & d \end{pmatrix}}
&= \begin{pmatrix} \reversal{d} & -\reversal{b} \\
 -\reversal{c} & \reversal{a} \end{pmatrix}.
\end{split}
\end{equation*}
\end{lemma}

\begin{definition}
For a matrix \( A = \begin{psmallmatrix} a & b \\ c & d \end{psmallmatrix}
\in \Mat(2, \clifford(V))\), define its {\em pseudo-determinant} as
    \(\Delta(A) = a\reversal{d} - b\reversal{c}\).
\end{definition}

Observe that, for all \(x \in V \oplus \real^{1,1}\),
\begin{equation}  \label{Q=pseudodeterminant}
\hat Q(x) = \Delta(j(x)).
\end{equation}

\subsection{Vahlen Matrices}

We introduce another key ingredient of this paper -- Vahlen matrices.

\begin{definition}
A matrix \(A \in \Mat(2, \clifford(V))\) is called a {\em Vahlen matrix}
if it is in the Lipschitz group \(\lipschitz(V \oplus \real^{1,1})\).
\end{definition}

Introduce a basis \(\set{n_0, n_\infty}\) for
\(\real^{1,1} \subset \Mat(2, \clifford(V))\):
\begin{equation*}
n_0 = \frac{e_- + e_+}{2}
= \begin{pmatrix} 0 & 0 \\ 1 & 0 \end{pmatrix} \quad\text{and}\quad
n_\infty = \frac{e_- - e_+}{2}
= \begin{pmatrix} 0 & 1 \\ 0 & 0 \end{pmatrix}.
\end{equation*}
Maks \cite{Maks.CliffordMobius} has a very useful characterization
of Vahlen matrices. We supply a proof of his criteria.

\begin{proposition}[\cite{Maks.CliffordMobius}]  \label{prop:Maks_criteria}
  If \(A = \begin{psmallmatrix} a & b \\ c & d \end{psmallmatrix}
  \in \Mat(2, \clifford(V))\), then
  \(A \in \lipschitz(V \oplus \real^{1,1})\) if and only if
  \begin{enumerate}
  \item \(a \conj{a}, \; b\conj{b}, \; c\conj{c}, \; d\conj{d} \in \real\),
  \item \(a\conj{c}, \; b\conj{d} \in V\),
  \item \(av \conj{b} - bv \conj{a}, \; cv \conj{d} - dv \conj{c} \in \real\)
    for all \(v \in V\),
  \item \(av \conj{d} - bv \conj{c} \in V\) for all \(v \in V\),
  \item \(a\reversal{b} = b\reversal{a}\), \(c\reversal{d} = d\reversal{c}\),
    and
  \item the pseudo-determinant
    \(\Delta(A) = a\reversal{d} - b\reversal{c}\) is a nonzero real number.
  \end{enumerate}
\end{proposition}

\begin{proof}
By Corollary \ref{cor:lipschitz_conjugate} and linearity,
it is sufficient to show that this list of conditions is equivalent to
\(A \conj{A} \in \real^\times\) and
\(A x \reversal{A} \in V \oplus \real^{1,1}\) for all
\(x \in V \cup \set{n_0, n_\infty}\).
Direct calculation shows
\begin{equation*}
\begin{split}
A n_0 \reversal{A} &=
\begin{pmatrix} a & b \\ c & d \end{pmatrix}
\begin{pmatrix} 0 & 0 \\ 1 & 0 \end{pmatrix}
\begin{pmatrix} \conj{d} & \conj{b} \\ \conj{c} & \conj{a} \end{pmatrix}
= \begin{pmatrix} b\conj{d} & b\conj{b} \\
  d\conj{d} & d\conj{b} \end{pmatrix}, \\
A n_{\infty} \reversal{A} &=
\begin{pmatrix} a & b \\ c & d \end{pmatrix}
\begin{pmatrix} 0 & 1 \\ 0 & 0 \end{pmatrix}
\begin{pmatrix} \conj{d} & \conj{b} \\ \conj{c} & \conj{a} \end{pmatrix}
= \begin{pmatrix} a\conj{c} & a\conj{a} \\
  c\conj{c} & c \conj{a} \end{pmatrix}, \\
A v \reversal{A} &=
\begin{pmatrix} a & b \\ c & d \end{pmatrix}
\begin{pmatrix} v & 0 \\ 0 & -v \end{pmatrix}
\begin{pmatrix} \conj{d} & \conj{b} \\ \conj{c} & \conj{a} \end{pmatrix}
= \begin{pmatrix} av\conj{d} - bv\conj{c} & av \conj{b} - bv \conj{a} \\
cv \conj{d} - dv \conj{c} & cv\conj{b} - dv \conj{a} \end{pmatrix}.
\end{split}
\end{equation*}
These expressions being in \(V \oplus \real^{1,1}\) for all \(v\in V\)
is equivalent to the first four criteria.
From the calculation
\begin{equation}  \label{det-calculation}
A \conj{A} =
\begin{pmatrix} a & b \\ c & d \end{pmatrix}
\Conj{\begin{pmatrix} a & b \\ c & d \end{pmatrix}}
= \begin{pmatrix} a & b \\ c & d \end{pmatrix}
\begin{pmatrix} \reversal{d} & -\reversal{b} \\
-\reversal{c} & \reversal{a} \end{pmatrix} \\
= \begin{pmatrix} a\reversal{d} - b\reversal{c} &
-a\reversal{b} + b \reversal{a} \\
c\reversal{d} - d \reversal{c} &
-\reversal{c} b + d \reversal{a} \end{pmatrix},
\end{equation}
we see that \(A \conj{A} \in \real^\times\) is equivalent to the last
two criteria.
\end{proof}

Recall the definition of the pin group \eqref{pin-def}.
Computation \eqref{det-calculation} implies the following
description of \(\Pin(V \oplus \real^{1,1})\) and a property of
the pseudo-determinants.

\begin{corollary}
The pin group \(\Pin(V \oplus \real^{1,1})\), as a subset of
\(\Mat(2, \clifford(V))\), consists of Vahlen matrices \(A\) with
the pseudo-determinant \(\Delta(A) = a\reversal{d} - b\reversal{c} = \pm 1\).
\end{corollary}

\begin{corollary}  \label{detAB}
  For Vahlen matrices \(A, B \in \Mat(2, \clifford(V))\), we have
  \(\Delta(AB) = \Delta(A) \Delta(B)\).
\end{corollary}

\begin{proof} Since \(A\) and \(B\) are Vahlen matrices,
\begin{equation*}
\Delta(A) \Delta(B) = A \conj{A} B \conj{B} = A (B \conj{B}) \conj{A}
 = AB \Conj{AB} = \Delta(AB).
\end{equation*}
\end{proof}

At this point, we would like to mention that Cnops \cite{Cnops.VahlenIndefinite}
has a more refined set of criteria for a matrix
\(A \in \Mat(2, \clifford(V))\) to be a Vahlen matrix.
%in the Lipschitz group \(\lipschitz(V \oplus \real^{1,1})\).
His criteria reduce to Ahlfors' criteria given in \cite{Ahlfors.MobiusClifford}
when \(V\) has positive definite signature.
(However, Cnops uses the sign convention \(v^2=Q(v)\) in his definition of
a Clifford algebra.)

\subsection{Conformal Space}

In order to relate Vahlen matrices to conformal transformations,
we observe that the embedding \eqref{conf-embedding} can be rewritten as
\begin{equation}\label{eq:conformal_embedding_matrix}
(j \circ \hat\iota)(v)
= \begin{pmatrix} v & -v^2 \\ 1 & -v \end{pmatrix}
= \begin{pmatrix} v \\ 1 \end{pmatrix}
\begin{pmatrix} 1 & -v \end{pmatrix}.
\end{equation}
This particular presentation encourages us to reinterpret the
twisted adjoint action \(\sigma\) of the Lipschitz group as an action on
a spinor-like object with two components.
We describe a construction of the conformal space due to
\Textcite{Maks.CliffordMobius} that makes this idea precise.

\begin{definition}
The {\em pre-conformal space} \(W_{\mathrm{pre}}\) is the set of products
\(\set{A e}\), where
\begin{equation*}
e = \frac{1 + e_- e_+}{2} =  \begin{pmatrix} 1 & 0 \\ 0 & 0 \end{pmatrix}
\end{equation*}
and \(A \in \Mat(2, \clifford(V))\) ranges over all Vahlen matrices.

The group of Vahlen matrices \(\lipschitz(V \oplus \real^{1,1})\) acts on
\(W_{\mathrm{pre}}\) by multiplication on the left.
\end{definition}

We see that any element in the pre-conformal space \(W_{\mathrm{pre}}\)
must have a matrix realization
\begin{equation}\label{eq:pre-conformal_matrix}
\begin{pmatrix} x & 0 \\ y & 0 \end{pmatrix}
\end{equation}
with entries \(x, y \in \clifford(V)\).
And since these entries form the first column of a Vahlen matrix,
by Proposition \ref{prop:Maks_criteria}, they satisfy
\(x\conj{x}, \: y\conj{y} \in \real\) and \(x\conj{y} \in V\).
Simplifying notations, we drop the right column and write \((x,y)\) or
\(\begin{psmallmatrix} x \\ y \end{psmallmatrix}\) for an element of
\(W_{\mathrm{pre}}\) represented by \eqref{eq:pre-conformal_matrix}.

In order to map the pre-conformal space \(W_{\mathrm{pre}}\) into the null cone
\(\mathcal{N}\) of \(V\oplus \real^{1,1}\) (recall equation \eqref{null-cone})
in a way that is compatible with \eqref{eq:conformal_embedding_matrix},
we introduce a map \(\gamma : W_{\mathrm{pre}} \to \Mat(2, \clifford(V))\)
\begin{equation}\label{eq:pre-conformal_to_null}
\gamma(Ae) = A e n_\infty \Reversal{A e} = A n_\infty \reversal{A}.
\end{equation}
Equivalently, we can use Lemma \ref{lmm:clifford_conjugate_matrix} and
express \eqref{eq:pre-conformal_to_null} in matrix form
\begin{equation}\label{eq:pre-conformal_to_null_matrix}
\gamma \begin{pmatrix} x \\ y \end{pmatrix}
= \begin{pmatrix} x & 0 \\ y & 0 \end{pmatrix}
\begin{pmatrix} 0 & 1 \\ 0 & 0 \end{pmatrix}
\begin{pmatrix} 0 & 0 \\ \conj{y} & \conj{x} \end{pmatrix}
= \begin{pmatrix} x\conj{y} & x\conj{x} \\
y\conj{y} & y\conj{x} \end{pmatrix}
= \begin{pmatrix} x\conj{y} & x\conj{x} \\
y\conj{y} & -x\conj{y} \end{pmatrix}.
\end{equation}
From this expression and equation \eqref{Q=pseudodeterminant}, we see that
\(\gamma\) takes values in the null cone \(\mathcal{N}\) of
\(V \oplus \real^{1,1}\).
%This map \(\gamma\) will be used to construct a bijection between the
%forthcoming conformal space \(W\) and the conformal compactification
%\(N(V)\) of \(V\).

For convenience we state the following result (see, for example,
\cite{Lam.IntroQuadraticForms}).

\begin{lemma}[Witt's Extension Theorem]
Let \( \mathcal V \) be a finite-dimensional real vector space
together with a non-degenerate symmetric bilinear form.
If \(\phi : {\mathcal V_1} \to {\mathcal V_2}\) is an isometric isomorphism
of two subspaces \({\mathcal V}_1, {\mathcal V}_2 \subset {\mathcal V}\),
then \(\phi\) extends to an isometric isomorphism
\(\hat{\phi} : {\mathcal V} \to {\mathcal V}\).
\end{lemma}

Recall that \(\varpi: V \oplus \real^{1,1} \setminus \{0\} \twoheadrightarrow
P(V \oplus \real^{1,1})\) is the quotient map.

\begin{lemma}\label{lmm:gamma_surjective}
The map \(\varpi \circ \gamma : W_{\mathrm{pre}} \to N(V)\),
which by abuse of notation we also denote by \(\gamma\), is surjective.
\end{lemma}
\begin{proof}
By the Witt's extension theorem, for each null vector \(x \in \mathcal{N}\),
there exists an orthogonal transformation in \(\Ogroup(V \oplus \real^{1,1})\)
that takes \(n_\infty\) to \(x\).
Since the twisted adjoint action \(\sigma\) is surjective onto
\(\Ogroup(V \oplus \real^{1,1})\), there exists a Vahlen matrix \(A_x\)
such that \(\sigma_{A_x}(n_\infty) = x\).
By equations \eqref{axa-reversal} and \eqref{eq:pre-conformal_to_null}, we have
\begin{equation*}
\gamma(A_x e) = A_x n_\infty \Reversal{A_x}
\sim A_x n_\infty \Gradeinv{A_x}^{-1} = \sigma_{A_x}(n_\infty) = x.
\end{equation*}
Therefore, it becomes an equality \(\gamma(A_xe) = x\) in \(N(V)\).
This proves \(\gamma : W_{\mathrm{pre}} \to N(V)\) is surjective.
\end{proof}

\begin{lemma}\label{lmm:upper_triangle}
Let \(A = \begin{psmallmatrix} a & b \\ c & d \end{psmallmatrix}\)
be a Vahlen matrix.
If \(\gamma(Ae)\) is proportional to \(n_\infty\),
then \(c = 0\) and \(a, d \in \lipschitz(V)\).
\end{lemma}

\begin{proof}
By equation \eqref{axa-reversal}, we have
\begin{equation*}
n_\infty \sim \gamma(Ae) = A n_\infty \reversal{A}
\sim A n_\infty \gradeinv{A}^{-1}.
\end{equation*}
Write \(A n_\infty \gradeinv{A}^{-1} = \lambda n_\infty\) for some
\(\lambda \in \real^\times\), then
\begin{equation*}
  \begin{pmatrix} 0 & a \\ 0 & c \end{pmatrix}
  = A n_\infty = \lambda n_\infty \gradeinv{A} =
  \lambda \begin{pmatrix} -\gradeinv{c} & \gradeinv{d} \\ 0 & 0 \end{pmatrix}.
\end{equation*}
We conclude that \(c=0\) and \(a = \lambda \gradeinv{d}\).
From Proposition \ref{prop:Maks_criteria} we see that
\(\Delta(A) = \lambda^{-1} a \conj{a} \in \real^\times\) and
\(\lambda^{-1} a v \reversal{a} \in V\) for all \(v\in V\).
Corollary \ref{cor:lipschitz_conjugate} implies \(a\) and
\(d = \lambda^{-1} \gradeinv{a} \in \lipschitz(V)\).
\end{proof}

\begin{lemma}\label{prop:gamma_kernel}
  Let \(A_1 = \begin{psmallmatrix} a_1 & b_1 \\ c_1 & d_1 \end{psmallmatrix}\)
  and \(A_2 = \begin{psmallmatrix} a_2 & b_2 \\ c_2 & d_2 \end{psmallmatrix}\)
  be two Vahlen matrices.
  Then \(\gamma(A_1 e)\) and \(\gamma(A_2 e)\) are proportional if and only if
  there exists an \(r \in \lipschitz(V)\) such that
  \((a_1, c_1) = (a_2 r, c_2 r)\).
\end{lemma}
\begin{proof}
Suppose first \((a_1, c_1) = (a_2 r, c_2 r)\) for some \(r \in\lipschitz(V)\),
then
\begin{equation*}
  \gamma(A_1 e) = \begin{pmatrix} a_1\conj{c}_1 & a_1\conj{a}_1  \\
    c_1 \conj{c}_1  & -a_1\conj{c}_1 \end{pmatrix}
  = r \conj{r} \begin{pmatrix} a_2\conj{c}_2 & a_2\conj{a}_2  \\
    c_2 \conj{c}_2  & -a_2\conj{c}_2 \end{pmatrix}
  = r \conj{r} \gamma(A_2 e).
\end{equation*}
Since \(r \conj{r} \in \real^\times\), we conclude that
\(\gamma(A_1 e)\) and \(\gamma(A_2 e)\) are proportional.

Conversely, suppose \(\gamma(A_1 e) = \lambda \gamma(A_2 e)\)
for some \(\lambda \in \real^\times\). Then
\begin{equation*}
  A_1 n_\infty \Reversal{A_1} = \gamma(A_1 e)
  = \lambda \gamma(A_2 e) = \lambda A_2 n_\infty \Reversal{A_2},
\end{equation*}
which implies
\begin{equation*}
  \gamma(A_2^{-1}A_1 e) = A_2^{-1} A_1 n_\infty \Reversal{A_2^{-1} A_1}
  = \lambda n_\infty.
\end{equation*}
Writing
\(  A_2^{-1} A_1 = \begin{psmallmatrix} a & b \\ c & d \end{psmallmatrix}\),
Lemma \ref{lmm:upper_triangle} implies that \(c = 0\) and
\(a\in \lipschitz(V)\), and thus
\begin{equation*}
A_2^{-1} A_1 e = \begin{pmatrix} a & b \\ 0 & d \end{pmatrix}
\begin{pmatrix} 1 & 0 \\ 0 & 0 \end{pmatrix}
= \begin{pmatrix} a & 0 \\ 0 & 0 \end{pmatrix}
= e a.
\end{equation*}
In other words, \(A_1 e = A_2 e a\), or in matrix form
\begin{equation*}
\begin{pmatrix} a_1 & 0 \\ c_1 & 0 \end{pmatrix}
= \begin{pmatrix} a_2 a & 0 \\ c_2 a & 0 \end{pmatrix}.
\end{equation*}
This finishes the proof.
\end{proof}

This lemma inspires the following definition of the conformal space.

\begin{definition}
The {\em conformal space} \(W\) of \(V\) is the pre-conformal space
\(W_{\mathrm{pre}}\) modulo the relation that \((x_1, y_1)\) and \((x_2, y_2)\)
are equivalent if and only if there exists an \(r \in \lipschitz(V)\)
such that \((x_1, y_1) = (x_2 r, y_2 r)\).
\end{definition}

\begin{theorem}
The map \(\gamma : W \to N(V)\) is well defined and is a bijection.
Moreover, \(\gamma\) intertwines the actions of
\(\lipschitz(V \oplus \real^{1,1})\), that is,
\(\gamma(A X) = \sigma_A(\gamma(X))\) for all Vahlen matrices
\(A\) and all \(X \in W\).
\end{theorem}

\begin{proof}
By Lemma \ref{prop:gamma_kernel}, \(\gamma\) is well defined and injective,
and, by Lemma \ref{lmm:gamma_surjective}, \(\gamma\) is surjective.

For an arbitrary element in \(W\) represented by \(X \in W_{\mathrm{pre}}\),
by equation \eqref{axa-reversal}, we have
\begin{equation*}
  \gamma(AX) = A X n_\infty \reversal{X} \reversal{A}
  \sim A X n_\infty \reversal{X} \gradeinv{A}^{-1}
  = \sigma_A( X n_\infty \reversal{X} ) = \sigma_A(\gamma(X)).
\end{equation*}
This becomes an equality in \(N(V)\) and proves that the two actions of
Vahlen matrices commute with \(\gamma\).
\end{proof}

\subsection{Geometry of the Conformal Space \(W\)}  \label{Geometry-subsection}

Recall that \cref{eq:conformal_embedding_matrix} was our inspiration for
the conformal space, so we wish to identify \(v \in V\) with \((v,1) \in W\),
but we first need to show that \((v,1) \in W\).
Indeed, \(A = \begin{psmallmatrix} v & 1 \\ 1 & 0 \end{psmallmatrix}\)
satisfies the conditions of Proposition \ref{prop:Maks_criteria},
thereby is a Vahlen matrix.
Geometrically, \(A\) is the composition of the inversion
\(\begin{psmallmatrix} 0 & 1 \\ 1 & 0 \end{psmallmatrix}\)
followed by the translation
\(\begin{psmallmatrix} 1 & v \\ 0 & 1 \end{psmallmatrix}\),
and so \(\gamma(Ae) = A n_\infty \reversal{A}\)
can be obtained from \(n_\infty\) by applying the inversion and
mapping \(n_\infty\) into \(n_0\), then translating the result by \(v\).

\Textcite{Maks.CliffordMobius} categorizes the points \((x,y)\) in \(W\)
into three classes.
\begin{itemize}
\item
  \(y\conj{y} \neq 0\).
  Then we have \(y^{-1} = \conj{y}/(y\conj{y})\) and
  \begin{equation*}
  \gamma \begin{pmatrix} x \\ y \end{pmatrix}
  = \begin{pmatrix} x\conj{y} & x\conj{x} \\
    y\conj{y}  & -x\conj{y} \end{pmatrix}
  \sim \begin{pmatrix} x y^{-1} & x\conj{x}(y\conj{y})^{-1} \\
        1 & -xy^{-1} \end{pmatrix}
  = \gamma \begin{pmatrix} xy^{-1} \\ 1 \end{pmatrix}.
\end{equation*}
  Thus, \((x,y)\) can be identified with \((xy^{-1}, 1)\) in \(W\)
  and in turn with \(xy^{-1}\) in \(V\).

\item
\(y \conj{y} = 0\) and \(x \conj{x} \neq 0\).
In this case, we have \(x^{-1} = \conj{x}/(x\conj{x})\) and
\begin{equation*}
\gamma \begin{pmatrix} x \\ y \end{pmatrix}
= \begin{pmatrix} -y\conj{x} & x\conj{x} \\
  y\conj{y}  & y\conj{x} \end{pmatrix}
\sim \begin{pmatrix} -yx^{-1} & 1 \\
  y\conj{y}(x\conj{x})^{-1} & yx^{-1} \end{pmatrix}
= \gamma \begin{pmatrix} 1 \\ yx^{-1} \end{pmatrix}.
\end{equation*}
Thus, \((x,y)\) can be identified with \((1, yx^{-1})\) in \(W\).
The Vahlen matrix that represents the inversion is
\(\begin{psmallmatrix} 0 & 1 \\ 1 & 0 \end{psmallmatrix}\),
and we recognize \((1, yx^{-1})\) as the inversion of \((yx^{-1}, 1)\).
In other words, \((x,y)\) represents the inversion of a null vector
\(yx^{-1} \in V\) and does not belong to \(V\).

\item
\(x \conj{x} = y \conj{y} = 0\) and \(x\conj{y} \ne 0\).
This class is empty in the Euclidean cases \(p=0\) or \(q=0\).
Otherwise, we can think of \((x,y)\) as the limiting point of the
line generated by the non-zero null vector \(x\conj{y} \in V\).
\end{itemize}

In terms of this identification of \(W\) with \(N(V)\) and embedding
\(V \hookrightarrow W\), the four types of conformal transformations
from Proposition \ref{types} can be represented by the following
Vahlen matrices:
\begin{enumerate}
\item
  parallel translations:
  \(\begin{psmallmatrix} 1 & b \\ 0 & 1 \end{psmallmatrix}\),
  where \(b \in V\);
\item
  orthogonal linear transformations:
  \(\begin{psmallmatrix} a & 0 \\ 0 & \gradeinv{a} \end{psmallmatrix}\),
  where \(a \in \lipschitz(V)\);
\item
  dilations:
  \(\begin{psmallmatrix} \sqrt{\lambda} & 0 \\
  0 & 1/\sqrt{\lambda} \end{psmallmatrix}\), where \(\lambda >0\);
\item
  the inversion:
  \(\begin{psmallmatrix} 0 & 1 \\ 1 & 0 \end{psmallmatrix}\).
\end{enumerate}
By Proposition \ref{types} and equation \eqref{Lipschitz-exact-seq},
every Vahlen matrix can be written as a finite product of matrices of these
four types.

Let \(U \subset V\) be a non-empty connected open subset.
When \(p+q>2\), by Theorem \ref{conformal-continuation-thm}, every conformal
transformation \(U \to \real^{p,q}\) can be described by a Vahlen matrix
acting through
\begin{equation}  \label{Vahlen-action}
\begin{pmatrix} a & b \\ c & d \end{pmatrix} x
= (ax+b)(cx+d)^{-1}
\end{equation}
where \(cx+d\) is invertible for all \(x\in U\).
Whenever we write a Vahlen matrix acting on \(x\) in this fashion,
we always assume that \(cx+d\) is invertible.

Before we go into the next section, we need to show that the product
\((x\reversal{c} + \reversal{d})(cx + d)\) is a real number.

\begin{lemma}  \label{cc-tilde}
  If \(A = \begin{psmallmatrix} a & b \\ c & d \end{psmallmatrix}\)
  is a Vahlen matrix, then \(\reversal{a} a,\: \reversal{b} b,\:
  \reversal{c} c,\: \reversal{d} d \in \real\).
\end{lemma}

\begin{proof}
By Corollary \ref{cor:lipschitz_conjugate_2}, the product
\begin{equation*}
  \conj{A} n_\infty A
  = \begin{pmatrix}
            \reversal{d} & -\reversal{b} \\
            -\reversal{c} & \reversal{a}
        \end{pmatrix}
        \begin{pmatrix}
            0 & 1 \\
            0 & 0
        \end{pmatrix}
        \begin{pmatrix}
            a & b \\
            c & d
        \end{pmatrix}
        =
        \begin{pmatrix}
            \reversal{d} c & \reversal{d} d \\
            -\reversal{c} c & -\reversal{c} d
        \end{pmatrix}
\end{equation*}
lies in \(V \oplus \real^{1,1}\),
which implies \(\reversal{c} c, \: \reversal{d} d \in \real\).
Similarly, \(\conj{A} n_0 A \in V \oplus \real^{1,1}\) leads to
\(\reversal{a} a, \: \reversal{b} b \in \real\).
\end{proof}

\begin{corollary}  \label{real-denom-cor}
  If \(x \in V\), the product
  \((x\reversal{c} + \reversal{d})(cx + d) \in \real\).  
\end{corollary}

\begin{proof}
Since \(cx+d\) is an entry of a Vahlen matrix
\begin{equation*}
    \begin{pmatrix}
        a & b \\ c & d
    \end{pmatrix}
    \begin{pmatrix}
        x & 1 \\ 1 & 0
    \end{pmatrix}
    =
    \begin{pmatrix}
        ax + b & a \\
        cx + d & c
    \end{pmatrix},
\end{equation*}
by Lemma \ref{cc-tilde}, \((x\reversal{c} + \reversal{d})(cx + d) \in \real\).
\end{proof}

Consequently, when \(cx+d\) is invertible, we have
\((x\reversal{c} + \reversal{d})(cx + d) \in \real^\times\).
It is worth mentioning that Maks \cite{Maks.CliffordMobius} claims
an even stronger result that if an entry in a Vahlen matrix
\(\begin{psmallmatrix} a & b \\ c & d \end{psmallmatrix}\) is invertible,
then that entry belongs to \(\lipschitz(V)\).

\section{Conformal Invariance of Monogenic Functions}  \label{invariance}

Monogenic function might not stay monogenic under translations by conformal
transformations.
That is, starting with a monogenic function \(f\) and a Vahlen matrix
\(A = \begin{psmallmatrix} a & b \\ c & d \end{psmallmatrix}\),
the composition function \(f(Ax)\) need not be monogenic, where
\(Ax\) is defined by \eqref{Vahlen-action}.
On the other hand, it is well known that in the positive definite case
(see \cite{Ryan.SingularitiesClifford, Bojarski1989}), the function
\begin{equation}  \label{eq:J_action}
  J_A(x) f(Ax) =
  \frac{(cx+d)^{-1}}{\absval{(x\reversal{c} + \reversal{d})(cx+d)}^{n/2-1}} f(Ax)
\end{equation}
is monogenic.
We extend this result to the case of the quadratic form \(Q\) on
the underlying vector space \(V \simeq \real^{p,q}\) having arbitrary signature.

\begin{lemma}
  Let \(A = \begin{psmallmatrix} a & b \\ c & d \end{psmallmatrix}\)
  be a Vahlen matrix, then
\begin{equation}\label{eq:transformation_difference}
\begin{split}
A x - A y &= \Delta(A) (y\reversal{c} + \reversal{d})^{-1} (x-y) (cx + d)^{-1} \\
&= \Delta(A) (x\reversal{c} + \reversal{d})^{-1} (x-y) (cy + d)^{-1},
\end{split}
\end{equation}
    for all \(x,y \in V\) such that \(cx+d\) and \(cy+d\) are invertible.
\end{lemma}

\begin{proof}
A proof by direct calculation is possible
(see \cite{Ahlfors.MobiusClifford, Ryan.ConformallyCovariant}).
Here, we give a proof by induction based on the fact that every
orthogonal transformation in \(\Ogroup(V \oplus \real^{1,1})\)
is a composition of translations, orthogonal transformations in \(\Ogroup(V)\),
dilations, and the inversion (Proposition \ref{types}).
It is straightforward to show that the formula holds for the Vahlen matrices
listed in Subsection \ref{Geometry-subsection} that produce translations,
orthogonal transformations, dilations, and the inversion on \(V\).
Then we show that if the formula is true for Vahlen matrices
\(A_1 = \begin{psmallmatrix} a_1 & b_1 \\ c_1 & d_1 \end{psmallmatrix}\) and
\(A_2 = \begin{psmallmatrix} a_2 & b_2 \\ c_2 & d_2 \end{psmallmatrix}\),
then the formula also holds for
\(A_{21}= A_2 A_1 =
\begin{psmallmatrix} a_{21} & b_{21} \\ c_{21} & d_{21} \end{psmallmatrix}\).
That is, we need to show
\begin{multline*}
  \Delta(A_{21}) (y \reversal{c}_{21} + \reversal{d}_{21})^{-1} (x-y)
  (c_{21} x + d_{21})^{-1}  \\
  = \Delta(A_2) \Delta(A_1) ((A_1 y) \reversal{c}_2 + \reversal{d}_2)^{-1}
  (y\reversal{c}_1 + \reversal{d}_1)^{-1} (x-y) (c_1x+d_1)^{-1}
  (c_2 (A_1 x) + d_2)^{-1}.
\end{multline*}
By Corollary \ref{detAB}, it is sufficient to show
\begin{equation*}
\begin{split}
(y \reversal{c}_{21} + \reversal{d}_{21})^{-1}
  &= ((A_1 y) \reversal{c}_2 + \reversal{d}_2)^{-1}
(y\reversal{c}_1 + \reversal{d}_1)^{-1}   \\
(c_{21} x + d_{21})^{-1} &= (c_1x+d_1)^{-1}(c_2 (A_1 x) + d_2)^{-1}.
\end{split}
\end{equation*}
These two identities are related to each other through the reversal operation,
so it is sufficient to show
\begin{equation} \label{eq:j_homomorphism}
  c_{21} x + d_{21} = (c_2 a_1 + d_2 c_1) x + (c_2 b_1 + d_2 d_1)
  = (c_2 (A_1 x) + d_2) (c_1x+d_1).
\end{equation}
This identity follows from \((A_1 x) (c_1x+d_1) = a_1 x + b_1\).

The second identity in \eqref{eq:transformation_difference} follows
from the first by applying the reversal.
\end{proof}

\begin{corollary}
  The partial derivatives of the conformal transformation produced by a
  Vahlen matrix \(A = \begin{psmallmatrix} a & b \\ c & d \end{psmallmatrix}\)
  can be expressed as
\begin{equation}\label{eq:transformation}
  \pdv{(A x)}{x_\mu} =
  \Delta(A) (x\reversal{c} + \reversal{d})^{-1} e_\mu (cx + d)^{-1} \\
  = \frac{\Delta(A)}{(x\reversal{c} + \reversal{d})(cx+d)} (cx+d) e_\mu
  (cx + d)^{-1}.
\end{equation}
\end{corollary}

\begin{proof}
Using equation \eqref{eq:transformation_difference}, we obtain
\begin{equation*}
\pdv{(A x)}{x_\mu} =  \lim_{h \to 0} \frac{A (x+h e_\mu) - Ax}h
  = \Delta(A) (x\reversal{c} + \reversal{d})^{-1} e_\mu (cx + d)^{-1}.
\end{equation*}
%\begin{multline*}
%  \pdv{(A x)}{x_\mu} =  \lim_{h \to 0} \frac{A (x+h e_\mu) - Ax}h  \\
%  = \lim_{h \to 0} \Delta(A) (x\reversal{c} + \reversal{d})^{-1} e_\mu
%  (c(x+he_{\mu}) + d)^{-1}
%  = \Delta(A) (x\reversal{c} + \reversal{d})^{-1} e_\mu (cx + d)^{-1}.
%\end{multline*}
\end{proof}

This expression for the partial derivatives allows us to verify directly
that the mapping \(x\mapsto Ax\) is a conformal transformation
(compare with \eqref{conformal-transf-condition}).

\begin{lemma}
The pull-back
\begin{equation*}
A^* B = \Omega_A^2 B,
\end{equation*}
where the conformal factor is
\begin{equation*}
  \Omega_A(x) = \frac{\Delta(A)}{(x\reversal{c} + \reversal{d})(cx+d)}.
\end{equation*}
\end{lemma}

\begin{proof}
  Write \(Ax = \sum_{\kappa=1}^n (Ax)_{\kappa}e_{\kappa}\), then,
  by \eqref{eq:transformation},
\begin{equation}  \label{eq:rearranged_transformation}
(cx+d) e_\mu (cx + d)^{-1}
%  = \frac{(x\reversal{c} + \reversal{d})(cx+d)}{\Delta(A)}\pdv{A x}{x_\mu}
= \frac{(x\reversal{c} + \reversal{d})(cx+d)}{\Delta(A)}
\sum_{\kappa=1}^n \pdv{(A x)_\kappa}{x_\mu} e_\kappa
\end{equation}
By \eqref{xy+yx}, we have
\begin{equation*}
-2 B(e_\mu, e_\nu) = (cx+d) (e_\mu e_\nu + e_\nu e_\mu) (cx + d)^{-1}.
\end{equation*}
Inserting \((cx+d)^{-1} (cx+d)\) between \(e_\mu\) and \(e_\nu\) and
applying \cref{eq:rearranged_transformation}, we obtain
\begin{equation}  \label{B-pullback-eqn}
\begin{split}
-2 B(e_\mu, e_\nu)
&= \left[ \frac{(x\reversal{c} + \reversal{d})(cx+d)}{\Delta(A)} \right]^2
\sum_{\kappa,\lambda=1}^n (e_\kappa e_\lambda + e_\lambda e_\kappa)
\pdv{(A x)_\kappa}{x_\mu} \pdv{(A x)_\lambda}{x_\nu}  \\
&= -2 \Omega_A(x)^{-2} \sum_{\kappa,\lambda=1}^n B(e_{\kappa},e_{\lambda})
\pdv{(A x)_\kappa}{x_\mu} \pdv{(A x)_\lambda}{x_\nu},
\end{split}
\end{equation}
and the result follows.
\end{proof}

\begin{corollary}
  Write \(\B\) for the matrix of the bilinear form with entries \(B(e_i,e_j)\)
  and \(\B^{-1}\) for its inverse. We have:
\begin{equation}  \label{part-deriv-sum}
  \sum_{\lambda,\mu=1}^n [\B^{-1}]_{\lambda\mu} e_\lambda (cx+d)^{-1}
  \pdv{(A x)_\kappa}{x_\mu}
  = \Omega_A(x) (cx+d)^{-1} \sum_{\nu=1}^n [\B^{-1}]_{\nu\kappa} e_{\nu}.
\end{equation}  
\end{corollary}

\begin{proof}
Let \(\partial A\) denote the matrix of partial derivatives with
\(\pdv{(A x)_i}{x_j}\) in the \((ij)\)-entry.
Then equation \eqref{B-pullback-eqn} can be rewritten as
\begin{equation*}
(\partial A)^{\tran} \B (\partial A) = \Omega_A^2 \B,
\end{equation*}
which in turn can be rewritten as
\begin{equation*}
(\partial A) \B^{-1} (\partial A)^{\tran} = \Omega_A^2 \B^{-1}
\end{equation*}
or, equivalently,
\begin{equation*}
\Omega_A^2(x) [\B^{-1}]_{\nu\kappa}
= \sum_{\lambda,\mu=1}^n [\B^{-1}]_{\lambda\mu}
\pdv{(A x)_\nu}{x_\lambda} \pdv{(A x)_\kappa}{x_\mu}.
\end{equation*}
Hence, using \eqref{eq:transformation},
\begin{equation*}
\begin{split}
\sum_{\nu=1}^n [\B^{-1}]_{\nu\kappa} e_{\nu}
&= \Omega_A^{-2}(x) \sum_{\lambda,\mu,\nu=1}^n [\B^{-1}]_{\lambda\mu}
\pdv{(A x)_\nu}{x_\lambda} \pdv{(A x)_\kappa}{x_\mu} e_{\nu}  \\
&= \Omega_A^{-2}(x) \sum_{\lambda,\mu=1}^n [\B^{-1}]_{\lambda\mu}
\pdv{(A x)}{x_\lambda} \pdv{(A x)_\kappa}{x_\mu}  \\
&= \Omega_A^{-1}(x) \sum_{\lambda,\mu=1}^n [\B^{-1}]_{\lambda\mu}
   (cx+d) e_\lambda (cx+d)^{-1} \pdv{(A x)_\kappa}{x_\mu},
\end{split}
\end{equation*}
and the result follows.
\end{proof}
  
To each Vahlen matrix
\(A = \begin{psmallmatrix} a & b \\ c & d \end{psmallmatrix}\)
we associate a function
\begin{equation*}
J_A(x) = \frac{(cx+d)^{-1}}{\absval{(x\reversal{c} + \reversal{d})(cx+d)}^{n/2-1}}
\end{equation*}
defined for all \(x \in V\) such that \(cx+d\) is invertible.
Note that, by Corollary \ref{real-denom-cor},
the expression \((x\reversal{c} + \reversal{d})(cx + d)\)
in the denominator is always real.

\begin{lemma}\label{lmm:inductive_monogenicity}
Given Vahlen matrices \(A_1\) and \(A_2\), we have
    \begin{equation*}
        J_{A_2 A_1}(x) = J_{A_1}(x) J_{A_2}(A_1 x)
    \end{equation*}
    wherever both sides are defined.
\end{lemma}

\begin{proof}
Let \(j_A(x) = cx + d\), then \(J_A\) can be expressed as
\begin{equation*}
J_A(x) = \frac{(j_A(x))^{-1}}{\absval{\Reversal{j_A(x)} j_A(x)}^{n/2-1}},
\end{equation*}
and equation \eqref{eq:j_homomorphism} can be restated as
\begin{equation*}
  j_{A_2 A_1}(x) = j_{A_2}(A_1x) j_{A_1}(x).
\end{equation*}
Therefore,
\begin{equation*}
\begin{split}
  J_{A_2 A_1}(x) &= \frac{(j_{A_2 A_1}(x))^{-1}}
  {\absval{\Reversal{j_{A_2 A_1}(x)} j_{A_2 A_1}(x)}^{n/2-1}} \\
  &= \frac{(j_{A_1}(x))^{-1}}{\absval{\Reversal{j_{A_1}(x)}j_{A_1}(x)}^{n/2-1}}
  \frac{(j_{A_2}(A_1 x))^{-1}}
   {\absval{\Reversal{j_{A_2}(A_1 x)} j_{A_2}(A_1 x)}^{n/2-1}} \\
&= J_{A_1}(x) J_{A_2}(A_1 x).
\end{split}
\end{equation*}
\end{proof}

\begin{theorem}  \label{main}
  Let \(U \subseteq V\) be an open set, \(\cal M\) a left
  \(\clifford(V)\)-module, and \(f: U \to {\cal M}\) a differentiable function.
  For each Vahlen matrix
  \(A = \begin{psmallmatrix} a & b \\ c & d \end{psmallmatrix}\),
\begin{equation}   \label{left-mono}
D\bigl( J_A(x) f(Ax) \bigr) = \Omega_A(x) J_A(x) (Df)(Ax)
\end{equation}
for all \(x \in A^{-1}(U) = \set{ v \in V ;\: Av \in U}\)
such that \(cx+d\) is invertible.
In particular, if \(f\) is left monogenic, so is \(J_A(x) f(A x)\).

Similarly, if \(\cal M'\) a right \(\clifford(V)\)-module,
and \(g: U \to {\cal M'}\) a differentiable function,
 \begin{equation}  \label{right-mono}
 \bigl( g(Ax) \Reversal{J_A} \bigr) D = \Omega_A(x) (gD)(x) \Reversal{J_A}(x)
 \end{equation}
 for all \(x \in A^{-1}(U) = \set{ v \in V ;\: Av \in U}\)
 such that \(cx+d\) is invertible, where
\begin{equation*}
  \Reversal{J_A}(x) =
  \frac{(x\reversal{c} + \reversal{d})^{-1}}
       {\absval{(x\reversal{c} + \reversal{d})(cx+d)}^{n/2-1}}.
\end{equation*}
In particular, if \(g\) is right monogenic, so is
\(g(A x) \Reversal{J_A}(x)\).
\end{theorem}

\begin{proof}
Choose a basis \(\set{e_1,\dots,e_n}\) of \(V\).
By equations \eqref{D-expression} and \eqref{part-deriv-sum},
\begin{equation}  \label{main-calculation}
\begin{split}
D \bigl( J_A(x) f(Ax) \bigr) &= (D J_A)(x) f(Ax)
+ \sum_{\kappa,\lambda,\mu=1}^n e_{\lambda} [\B^{-1}]_{\lambda\mu} J_A(x)
\pdv{(Ax)_{\kappa}}{x_{\mu}} \pdv{f}{x_{\kappa}}  \\
&= (D J_A)(x) f(Ax)
+ \Omega_A(x) J_A(x) \sum_{\kappa,\nu=1}^n e_{\nu} [\B^{-1}]_{\nu\kappa}
\pdv{f}{x_{\kappa}}  \\
&= (D J_A)(x) f(Ax) + \Omega_A(x) J_A(x) (Df)(Ax).
\end{split}
\end{equation}
Thus, it remains to prove that \(J_A(x)\) is left monogenic wherever
it is defined.

If \(A\) is one of the Vahlen matrix listed in
  Subsection \ref{Geometry-subsection} that produces a translation,
  orthogonal linear transformation, dilation on \(V\), then \(J_A\)
  is just a constant function, hence monogenic.
If \(A=\begin{psmallmatrix} 0 & 1 \\ 1 & 0 \end{psmallmatrix}\) represents
the inversion on \(V\), then
\begin{equation}  \label{J-inversion}
J_A(x) = \sgn(x^2) \frac{x}{\absval{x^2}^{n/2}}.
\end{equation}
    By direct calculation,
    \begin{equation*}
Dx=-n, \qquad D(x^2) = -2x, \qquad D \absval{x^2} = -2\sgn(x^2)x,
    \end{equation*}
\begin{equation*}
  D \biggl( \sgn(x^2) \frac{x}{\absval{x^2}^{n/2}} \biggr)
  = -\sgn(x^2) \frac{n}{\absval{x^2}^{n/2}} + \frac{nx^2}{\absval{x^2}^{n/2+1}} =0,
\end{equation*}
and \(J_A(x)\) is monogenic too.

By equation \eqref{main-calculation}, \(J_A(x) f(A x)\) is left monogenic
whenever \(f\) and \(J_A\) are left monogenic.
Thus, by Lemma \ref{lmm:inductive_monogenicity}, \(J_{A_2A_1}\) is
left monogenic whenever \(J_{A_1}\) and \(J_{A_2}\) are left monogenic.
Since each Vahlen matrix can be written as a finite product
of Vahlen matrices realizing translations, orthogonal transformations,
dilations and the inversion on \(V\),
this proves that \(J_A\) is left monogenic for all Vahlen matrices \(A\).
This finishes the proof of \eqref{left-mono}.

The proof of \eqref{right-mono} is similar.
\end{proof}

\appendix

\section{A Note on the Definition of the Dirac Operator} \label{appendix}

We introduced the Dirac operator with plus and minus signs in
\eqref{Dirac-operator}, and we have shown that this is a natural construction
from the point of view of basis independence.
Since this phenomenon does not happen in the positive definite case,
it is perhaps worthwhile to show that the classical formula
\cref{eq:J_action} fails if the Dirac operator were defined differently,
for example, having all positive signs:
\begin{equation*}
D^* = e_1 \pdv{}{x_1} + \cdots + e_p \pdv{}{x_p} + e_{p+1} \pdv{}{x_{p+1}}
+ \cdots + e_{p+q} \pdv{}{x_{p+q}}.
\end{equation*}
Note that this differential operator depends on the choice of basis
-- a different choice of orthogonal basis of \(V \simeq \mathbb{R}^{p,q}\)
satisfying \eqref{V-basis} typically leads to a different operator.
This can be seen by trying to adapt the proof of
Lemma \ref{D-basis-independence} for \(D^*\).

As a first example, consider \(\real^{2,1}\) with orthogonal generators of
\(\clifford(\real^{2,1})\) satisfying \(e_1^2 = -1\), \(e_2^2 = -1\) and
\(e_3^2 = 1\).
The function \(f(x) = x_1e_1 - x_2e_2\) is in the kernel of both
\(D^*\) and \(D\), but if we consider a Vahlen matrix
\begin{equation*}
 A = \begin{pmatrix} \cosh \alpha + e_2 e_3 \sinh \alpha & 0 \\
   0 & \cosh \alpha + e_2 e_3 \sinh \alpha \end{pmatrix}
\end{equation*}
producing a hyperbolic rotation in the \(e_2e_3\)-plane inside \(\real^{2,1}\),
we find that \(J_A(x) f(Ax)\) is not in the kernel of \(D^*\).
Indeed, by direct computation,
\begin{equation*}
  J_A(x) = \cosh \alpha - e_2 e_3 \sinh \alpha,
\end{equation*}
\begin{equation*}
  A e_1 = e_1, \qquad
  A e_2 = e_2 \cosh (2\alpha) + e_3 \sinh (2\alpha), \qquad
  A e_3 = e_2 \sinh (2 \alpha) + e_3 \cosh (2\alpha).
\end{equation*}
Therefore, we have
\begin{equation*}
\begin{split}  
  J_A(x) f(Ax) &= (\cosh \alpha - e_2 e_3 \sinh \alpha)
  \bigl( x_1 e_1 - ( x_2 \cosh(2\alpha) + x_3 \sinh (2\alpha)) e_2 \bigr)  \\
  &= x_1 e_1 (\cosh \alpha - e_2 e_3 \sinh \alpha)
  - ( x_2 \cosh (2\alpha) + x_3 \sinh (2\alpha) )
  ( e_2 \cosh \alpha - e_3 \sinh \alpha).
\end{split}
\end{equation*}
Then \(D^*\) acting on the first term yields
\begin{equation*}
  D^* \bigl[ x_1 e_1 (\cosh \alpha - e_2 e_3 \sinh \alpha) \bigr]
  = -(\cosh \alpha - e_2 e_3 \sinh \alpha),
\end{equation*}
whereas the second term results in
\begin{equation*}
  D^* \bigl[ - ( x_2 \cosh (2\alpha) + x_3 \sinh (2\alpha) )
    ( e_2 \cosh \alpha - e_3 \sinh \alpha) \bigr]
  = \cosh (3\alpha) + e_2 e_3 \sinh (3\alpha).
\end{equation*}
They clearly do not cancel each other.
To contrast, if we use the true Dirac operator,
\begin{equation*}
\begin{split}
  &D \bigl[ x_1 e_1 (\cosh \alpha - e_2 e_3 \sinh \alpha) \bigr]
  = -(\cosh \alpha - e_2 e_3 \sinh \alpha),  \\
  &D \bigl[ - ( x_2 \cosh (2\alpha) + x_3 \sinh (2\alpha) )
    ( e_2 \cosh \alpha - e_3 \sinh \alpha) \bigr]
  = \cosh \alpha - e_2 e_3 \sinh \alpha,
\end{split}
\end{equation*}
as desired.

In general, the true Dirac operator \(D\) plays well with orthogonal linear
transformations, while the ostensible Dirac operator \(D^*\) does not.
More precisely, restricting to the Vahlen matrices producing orthogonal
linear transformations on \(V\), that is, matrices of the form
\(\begin{psmallmatrix} a & 0 \\ 0 & \gradeinv{a} \end{psmallmatrix}\),
where \(a \in \lipschitz(V)\), equation \eqref{eq:J_action} becomes
\begin{equation*}
  J_A(x) f(Ax) =
  \frac{\gradeinv{a}^{-1}}{\absval{\reversal{a}a}^{n/2-1}} f(a x \gradeinv{a}^{-1})
\sim \gradeinv{a}^{-1} f(a x \gradeinv{a}^{-1}).
\end{equation*}
It was established in Theorem \ref{main} that
\begin{equation*}
  Df=0 \qquad \text{implies} \qquad
  D \bigl( \gradeinv{a}^{-1} f(a x \gradeinv{a}^{-1}) \bigr) = 0.
\end{equation*}
On the other hand, as the above example shows,
\begin{equation*}
  \text{when} \quad D^*f=0, \quad \text{typically} \quad
  D^* \bigl( \gradeinv{a}^{-1} f(a x \gradeinv{a}^{-1}) \bigr) \ne 0.
\end{equation*}

The ostensible Dirac operator \(D^*\) also does not behave well with respect
to the inversion. The constant function \(f(x)=1\) is annihilated by \(D^*\)
(and \(D\)). Applying to this function the inversion realized by
\(A=\begin{psmallmatrix} 0 & 1 \\ 1 & 0 \end{psmallmatrix}\)
results in \(J_A(x)\), and \(J_A(x)\) is not in the kernel of \(D^*\).
For concreteness, consider \(\real^{1,1}\) with generators of
\(\clifford(\real^{1,1})\) satisfying \(e_1^2 = -1\), \(e_2^2 = 1\).
By equation \eqref{J-inversion}, we have
\begin{equation*}
J_A(x) = \sgn(x^2) \frac{x}{\absval{x^2}}
= \frac{x_1e_1 + x_2e_2}{(x_2)^2 - (x_1)^2}.
\end{equation*}
Then \(D^*[x_1e_1 + x_2e_2] = 0\), and
\begin{equation*}
  D^* J_A(x) %= D^* \biggl( \frac{x_1e_1 + x_2e_2}{(x_2)^2 - (x_1)^2} \biggr)
  = 2 \frac{(x_1e_1-x_2e_2)(x_1e_1 + x_2e_2)}{((x_2)^2 - (x_1)^2)^2}
  = 2 \frac{-(x_1)^2-(x_2)^2+2x_1x_2e_1e_2}{((x_2)^2 - (x_1)^2)^2},
\end{equation*}
which is clearly non-zero.

%\bibliographystyle{ieeetr}
%\bibliography{References.bib}

\printbibliography

\end{document}